\documentclass[12pt, a4paper]{article}
\usepackage{amssymb, amsmath,amsthm }
\usepackage[all]{xy}
\newtheorem{prop}{Proposition}[section]
\newtheorem{defi}[prop]{Definition}

\newtheorem{lem}[prop]{Lemma}
\newtheorem{thm}[prop]{Theorem}

\newtheorem{remar}[prop]{Remark}
\newtheorem{cor}[prop]{Corollary}

\DeclareMathAlphabet{\mathpzc}{OT1}{pzc}{m}{it}

\DeclareMathOperator{\End}{End}
\DeclareMathOperator{\Hom}{Hom}
\DeclareMathOperator{\Ind}{Ind}
\DeclareMathOperator{\cInd}{c-Ind}
\DeclareMathOperator{\Res}{Res}

\DeclareMathOperator{\GL}{GL}
\DeclareMathOperator{\SL}{SL}
\DeclareMathOperator{\Ker}{Ker}

\DeclareMathOperator{\Image}{Im}

\DeclareMathOperator{\Sym}{Sym}
\DeclareMathOperator{\Rep}{Rep}

\DeclareMathOperator{\supp}{Supp}
\DeclareMathOperator{\id}{id}
\DeclareMathOperator{\Mod}{Mod}

\DeclareMathOperator{\val}{val}
\DeclareMathOperator{\Inj}{Inj}
\DeclareMathOperator{\inj}{inj}
\DeclareMathOperator{\soc}{soc}

\DeclareMathOperator{\Ext}{Ext}
\DeclareMathOperator{\RR}{\mathbb{R}}

\DeclareMathOperator{\ind}{ind}
\DeclareMathOperator{\detr}{det}
\DeclareMathOperator{\Spec}{Spec}
\newcommand{\cIndu}[3]{\cInd_{#1}^{#2}{#3}}

\newcommand{\Indu}[3]{\Ind_{#1}^{#2}{#3}}

\newcommand{\FF}{\mathbb F}

\newcommand{\Fp}{\mathbb{F}_{p}}
\newcommand{\Fbar}{\overline{\mathbb{F}}_{p}}

\newcommand{\pF}{\mathfrak{p}}

\newcommand{\NN}{\mathbb{N}}

\newcommand{\VV}{\mathbf{V}}

\newcommand{\VVV}{\mathcal V}

\newcommand{\Qp}{\mathbb {Q}_p}

\newcommand{\Zp}{\mathbb{Z}_p}

\newcommand{\HH}{\mathcal H}

\newcommand{\Eins}{\mathbf 1}
\newcommand{\St}{St}
\newcommand{\Sp}{\mathrm{Sp}}
\newcommand{\II}{\mathcal I}
\newcommand{\TT}{\mathcal T}

\newcommand{\gal}{\mathcal G_{\Qp}}
\newcommand{\tsigma}{\tilde{\sigma}}
\newcommand{\OO}{\mathcal O}
\newcommand{\MM}{\mathfrak m}
\newcommand{\oE}{\OO_E}
\newcommand{\GG}{\mathcal G}
\title{Extensions for supersingular representations of $\GL_2(\Qp)$}
\begin{document} 

\author{Vytautas Pa\v{s}k\={u}nas}

\maketitle

\abstract{\noindent Let $p>2$ be a prime number. Let $G:=\GL_2(\Qp)$ and $\pi$, $\tau$ smooth irreducible 
representations of $G$ on $\Fbar$-vector spaces with a central character. We show if $\pi$ is supersingular then 
 $\Ext^1_G(\tau,\pi)\neq 0$ implies $\tau\cong \pi$ and compute the dimension of $\Ext^1_G(\pi, \pi)$. This answers affirmatively for $p>2$ 
a question of Colmez.
We also determine $\Ext^1_G(\tau,\pi)$, when $\pi$ is the Steinberg representation. As a consequence of 
our results combined with those already in the literature one knows extensions between  all the irreducible representations 
of $G$.}

\section{Introduction}
\parindent0mm
\parskip5pt
In this paper we study the category $\Rep_G$ of smooth representations of $G:=\GL_2(\Qp)$ on $\Fbar$-vector spaces.
Smooth irreducible $\Fbar$-representations of $G$ with a central character have been classified by Barthel-Livne \cite{bl}
and Breuil \cite{b1}. A smooth irreducible representation $\pi$ of $G$ is supersingular, if 
it is not a subquotient of any principal series representation. Roughly speaking a supersingular 
representation is an $\Fbar$-analog of a supercuspidal representation. 
 
\begin{thm}\label{A} Assume that $p>2$ and let $\tau$ and $\pi$ be irreducible smooth representations of $G$ admitting a central character. 
If  $\pi$ is supersingular and $\Ext^1_G(\tau, \pi)\neq 0$ then $\tau\cong \pi$. Moreover, if $p\ge 5$ then $\dim \Ext^1_G(\pi, \pi)=5$.
\end{thm}

This answers affirmatively for $p>2$ a question of Colmez, see the introduction of \cite{col1}.
 When $p=3$ there are two cases and we can show that in one of them 
$\dim \Ext^1_G(\pi, \pi)=5$, in the other $\dim \Ext^1_G(\pi, \pi)\le 6$, which is the expected dimension. 
We  note that if $\tau$ is a twist of Steinberg representation by a character or irreducible principal series 
then  Colmez \cite[VII.5.4]{col1} and Emerton \cite[Prop. 4.2.8]{em} prove by  different methods that $\Ext^1_G(\tau, \pi)=0$. Our result is new 
when $\tau$ is supersingular or a character.

We now explain the strategy of the proof. We first get rid of the extensions coming from the centre $Z$ of $G$. Let $\zeta: Z\rightarrow \Fbar^{\times}$ be the 
central character of $\pi$, and let $\Rep_{G,\zeta}$ be the full subcategory of $\Rep_G$ consisting of representations with the central character 
$\zeta$. We show in Theorem \ref{centre} that if $\Ext^1_G(\tau,\pi)\neq 0$ then $\tau$ also admits a central character $\zeta$. Let 
$\Ext^1_{G,\zeta}(\tau,\pi)$ parameterise all the isomorphism classes of extensions between $\pi$ and $\tau$ admitting a central character $\zeta$.
We show that if $\tau\not \cong \pi$ then $\Ext^1_{G,\zeta}(\tau,\pi)\cong \Ext^1_G(\tau,\pi)$ and there exists an exact sequence:
\begin{equation}\label{C}
0\rightarrow \Ext^1_{G,\zeta}(\pi,\pi)\rightarrow \Ext^1_G(\pi,\pi)\rightarrow \Hom(Z,\Fbar)\rightarrow 0,
\end{equation}  
where $\Hom$ denotes continuous group homomorphisms. Let $I$ be the `standard' Iwahori 
subgroup of $G$, (see \S2), and $I_1$ the maximal pro-$p$ subgroup of $I$. Since $\zeta$ is smooth, it is trivial on 
the pro-$p$ subgroup $I_1\cap Z$, hence we may consider $\zeta$ as a character of $ZI_1$.
Let $\HH:=\End_G(\cIndu{ZI_1}{G}{\zeta})$ be the Hecke algebra,
and $\Mod_{\HH}$ the category of right $\HH$-modules. Let $\II: \Rep_{G,\zeta}\rightarrow \Mod_{\HH}$ be the functor
$$\II(\kappa):=\kappa^{I_1}\cong \Hom_{G}( \cIndu{ZI_1}{G}{\zeta},\kappa).$$ 
Vign\'{e}ras shows in \cite{vig} that $\II$ induces a bijection between irreducible representations of $G$ with the central character $\zeta$ 
and irreducible $\HH$-modules. Using results of Ollivier \cite{o2} we show that there exists an $E_2$-spectral sequence:
\begin{equation}\label{specseq1}
\Ext^i_{\HH}(\II(\tau), \RR^j \II(\pi))\Longrightarrow \Ext^{i+j}_{G, \zeta}(\tau, \pi).
\end{equation}
The $5$-term sequence associated to \eqref{specseq1} gives an exact sequence:
\begin{equation}\label{R} 
0\rightarrow \Ext^1_{\HH}(\II(\tau),\II(\pi))\rightarrow \Ext^1_{G,\zeta}(\tau,\pi)\rightarrow \Hom_{\HH}(\II(\tau),\RR^1\II(\pi)).
\end{equation}   
Now $\Ext^1_{\HH}(\II(\tau),\II(\pi))$ has been 
determined in \cite{bp} and in fact  is zero if $\tau\not \cong \pi$. The problem is to understand $\RR^1\II(\pi)$ as an $\HH$-module.

We have two approaches to this. Results of Kisin \cite{kis1} imply that the dimension of $ \Ext^1_G(\pi, \pi)$ is 
bounded below by the dimension of $\Ext^1_{\gal}(\rho, \rho)$, where 
$\rho$  is the $2$-dimensional irreducible 
$\Fbar$-rep\-re\-sen\-ta\-tion of $\gal$, the absolute Galois group of $\Qp$,
corresponding to $\pi$ under the mod $p$ Langlands, see \cite{b2}, \cite{col1}. (Excluding one case when $p=3$.)
Let $\mathfrak{I}$ be the image of 
$\Ext^1_{G,\zeta}(\pi,\pi)\rightarrow\Hom_{\HH}(\II(\pi),\RR^1\II(\pi))$. Using \eqref{C} and \eqref{R} we 
obtain a lower bound on  the dimension of $\mathfrak{I}$. By forgetting the $\HH$-module structure  
 we obtain an isomorphism of vector spaces:
$$ \RR^1\II(\pi)\cong H^1(I_1/Z_1,\pi),$$ 
where $Z_1$ is the maximal pro-$p$ subgroup of $Z$. The key idea is to bound the dimension of $H^1(I_1/Z_1,\pi)$ from above and use this to show 
if $\II(\tau)$ was a submodule of  $\RR^1\II(\pi)$ for some $\tau\not \cong \pi$ then this would force the  dimension of $\mathfrak{I}$ 
to be smaller than calculated before.

At the time of writing (an n-th draft of) this,  \cite{kis1} was not written up  and there were some technical issues with 
the outline of the argument in the introductions of \cite{col1} and \cite{kis}, 
 caused by the fact that all the representations in \cite{col1} are assumed to have a central character. 
Since we only need a lower bound on the dimension
of $\Ext^1_G(\pi, \pi)$ and only in the supersingular case, we have written up the proof of a weaker statement in the appendix. 
The proof given there is a variation on Colmez-Kisin argument. 

In order to  bound the dimension of $H^1(I_1/Z_1,\pi)$ we prove a new result about the structure of supersingular representations of $G$.
 Let $M$ be the subspace of $\pi$ generated 
by $\pi^{I_1}$ and the semi-group  $\bigl (\begin{smallmatrix} p^{\NN} & \Zp\\ 0 & 1\end{smallmatrix}\bigr )$.
One may show that $M$ is a representation of $I$. 

\begin{thm}\label{exseq1} The map $(v,w)\mapsto v-w$ induces an exact sequence of $I$-representations:
$$0\rightarrow \pi^{I_1}\rightarrow M\oplus \Pi\centerdot M\rightarrow \pi\rightarrow 0,$$
where $\Pi=\begin{pmatrix} 0 & 1\\ p & 0\end{pmatrix}$. 
\end{thm}

 We show that the restrictions of $M$ and $M/\pi^{I_1}$  to $I\cap U$, 
where $U$ is the unipotent upper triangular matrices, are injective objects in $\Rep_{I\cap U}$. If $\psi:I\rightarrow \Fbar^{\times}$ is 
a smooth character and $p>2$, using this,  we work out $\Ext^1_{I/Z_1}(\psi, M)$ and $\Ext^1_{I/Z_1}(\psi, M/\pi^{I_1})$. Theorem \ref{exseq1}
enables us to determine $H^1(I_1/Z_1,\pi)$ as a representation of $I$, see Theorem \ref{main} and Corollary \ref{dimH1}. Once one has this
it is quite easy to work out $\RR^1\II(\pi)$ as an $\HH$-module in the regular case, see Proposition \ref{R1regp5}, without using
Colmez's work. It is also possible to work out directly the $\HH$-module structure of $\RR^1\II(\pi)$ in the Iwahori case. However, 
the proof relies on   heavy calculations of $\Ext^1_K(\Eins, \pi)$ and $\Ext^1_K(\St, \pi)$, where $K:=\GL_2(\Zp)$ and $\St$ is the Steinberg representation 
of $K/K_1\cong \GL_2(\Fp)$. So we decided to exclude it and use ``strat\'egie de Kisin'' instead.

The primes $p=2$, $p=3$ require some special attention. Theorem \ref{exseq1} holds  when $p=2$, but our calculation 
of $H^1(I_1/Z_1, \pi)$ breaks down for the technical reason that  
the trivial character is the only smooth character of $I$, when $p=2$. However, if $p=2$ and we fix a 
central character $\zeta$ then there exists only one supersingular representation (up to isomorphism) with central character $\zeta$. 
Hence, it is enough to show that $\Ext^1_G(\tau,\pi)=0$ when $\tau$ is a character, since all the other cases are handled in 
\cite[VII.5.4]{col1}, \cite[\S4]{em}. It might be easier to do this directly.

Let $\Sp$ be the Steinberg representation of $G$. After the first draft of this paper, it was pointed out to me by Emerton that  
it was not known (although expected) that $\Ext^1_G(\eta, \Sp)=0$, when $\eta:G\rightarrow \Fbar^{\times}$ is a smooth  
character of order $2$ (all the other cases have been worked out in \cite[\S4]{em}, see also \cite[\S VII.4,\S VII.5]{col1}). 
A slight modification of  our proof for supersingular representations also 
works for the Steinberg representation. In the last section we work out $\Ext^1_G(\tau, \Sp)$ for all irreducible $\tau$, when $p>2$.
As a result of this and the results already in the literature (\cite{bp}, \cite{col1}, \cite{em}), one knows $\Ext^1_G(\tau, \pi)$   
for all irreducible $\tau$ and $\pi$, when $p>2$. We record this in the last section.

\textit{Acknowledgements.} The key ideas of this paper stem from a joint work with Christophe Breuil \cite{bp}. 
I would like to thank Pierre Colmez for pointing out this problem to me and Florian Herzig for a number of stimulating discussions.  
I would like to thank Ga\"etan Chenevier,  Pierre Colmez  and Mark Kisin for some very helpful discussion on the 
``strat\'egie de Kisin'' outlined in \cite{col1} and \cite{kis}. This paper was written when I was visiting 
Universit\'e Paris-Sud and IH\'ES, supported by Deutsche Forschungsgemeinschaft. I 
would like to thank these institutions.

\section{Notation}

Let $G:=\GL_2(\Qp)$, let $P$ be the subgroup of upper-triangular matrices, $T$ the subgroup of diagonal matrices, $U$ be the 
unipotent upper triangular matrices and $K:=\GL_2(\Zp)$. Let $\pF:=p\Zp$ and  
$$ I:=\begin{pmatrix} \Zp^{\times} & \Zp \\ \pF & \Zp^{\times} \end{pmatrix},\quad I_1:= \begin{pmatrix} 1+\pF & \Zp \\ \pF & 1+\pF \end{pmatrix}, 
\quad K_1:=\begin{pmatrix} 1+\pF & \pF \\ \pF & 1+\pF \end{pmatrix}.$$
For $\lambda\in \Fp$ we denote the Teichm\"uller lift of $\lambda$ to $\Zp$ by $[\lambda]$.  Set
$$H:=\biggl \{\begin{pmatrix} [\lambda] & 0\\ 0 & [\mu]\end{pmatrix}: \lambda, \mu\in \Fp^{\times}\biggr \}.$$  
Let $\alpha:H\rightarrow \Fbar^{\times}$ be the character
$$\alpha(\begin{pmatrix} [\lambda] & 0\\ 0 & [\mu]\end{pmatrix}):=\lambda\mu^{-1}.$$
Further, define 
$$\Pi:=\begin{pmatrix} 0 & 1\\ p & 0\end{pmatrix}, \quad s:=\begin{pmatrix} 0 & 1 \\ 1 & 0\end{pmatrix}, \quad 
t:= \begin{pmatrix} p & 0 \\ 0 & 1 \end{pmatrix}. $$
For $\lambda\in \Fbar^{\times}$ we define an unramified character
$\mu_{\lambda}:\Qp^{\times} \rightarrow \Fbar^{\times}$, by  $x\mapsto \lambda^{\val(x)}$.

Let $Z$ be the centre of $G$, and set $Z_1:=Z\cap I_1$. Let $G^0:=\{g\in G: \detr g\in \Zp^{\times}\}$ and set $G^+:=ZG^0$.

Let  $\mathcal G$ be a topological group. We denote by  $\Hom(\mathcal G, \Fbar)$ the continuous group homomorphism, where the additive group $\Fbar$ is 
given the  discrete topology. If $\VVV$ is a representation of $\mathcal G$ and $S$ is a subset of $\VVV$ we denote by 
$\langle \mathcal G \centerdot S\rangle$ the smallest subspace of $\VVV$ stable under the action of $\mathcal G$ and containing $S$.
Let $\Rep_{\mathcal G}$ be the category of smooth representations of $\mathcal G$ on $\Fbar$-vector spaces. If $\mathcal Z$ 
is the centre of $\mathcal G$ and $\zeta: \mathcal Z\rightarrow \Fbar^{\times}$ is a smooth character then we denote by $\Rep_{\mathcal G,\zeta}$ the 
full subcategory of $\Rep_{\mathcal G}$ consisting of representations with central character $\zeta$. 

All the representations in this paper are on $\Fbar$-vector spaces.

\section{Irreducible representations of $K$}
We recall some facts about the irreducible representations of $K$ and introduce some notation. Let $\sigma$ be an irreducible smooth 
representation of $K$.  Since $K_1$ is an open  pro-$p$ subgroup of $K$, the space of $K_1$-invariants $\sigma^{K_1}$ is non-zero,
and since $K_1$ is normal in $K$, $\sigma^{K_1}$ is a non-zero $K$-subrepresentation of $\sigma$, and since $\sigma$ is irreducible 
we obtain $\sigma^{K_1}=\sigma$. Hence the smooth irreducible representations of $K$ coincide with the irreducible
representations of $K/K_1\cong \GL_2(\Fp)$, and so there exists a uniquely determined pair of integers $(r,a)$ with 
$0\le r\le p-1$, $0\le a<p-1$, such that 
$$\sigma\cong \Sym^r \Fbar^2\otimes \detr^a.$$
Note that $r=\dim\sigma-1$ and throughout the paper given $\sigma$, $r$ will always mean $\dim \sigma -1$. The space of $I_1$-invariants
$\sigma^{I_1}$ is $1$-dimensional and so $H$ acts on $\sigma^{I_1}$ be a character $\chi_{\sigma}=\chi$. Explicitly,
$$\chi(\begin{pmatrix} [\lambda] & 0\\0 & [\mu]\end{pmatrix})= \lambda^r (\lambda\mu)^a.$$
We define an involution $\sigma\mapsto \tilde{\sigma}$ on the set of isomorphism classes of smooth irreducible representations of $K$ by 
setting 
$$\tilde{\sigma}:=\Sym^{p-r-1}\Fbar^2\otimes \detr^{r+a}.$$
Note that $\chi_{\tilde{\sigma}}=\chi_{\sigma}^s$. For the computational purposes it is convenient to identify $\Sym^r\Fbar^2$ with the space 
of homogeneous polynomials in $\Fbar[x,y]$ of degree $r$. The action of $\GL_2(\Fp)$ is given by
$$\begin{pmatrix} a & b \\ c& d\end{pmatrix} \centerdot P(x,y):= P(ax+cy, bx+dy).$$
With this identification $\sigma^{I_1}$ is spanned by $x^r$.
\begin{lem}\label{calcsym} let $0\le j\le r$ be an integer and define  $f_j\in \Sym^r\Fbar^2\otimes \det^a$ by
$$ f_j:=\sum_{\lambda\in \Fp} \lambda^{p-1-j}\begin{pmatrix} 1  & \lambda\\ 0 & 1\end{pmatrix} s x^r.$$
If $r=p-1$ and $j=0$ then $f_0=(-1)^{a+1}(x^r+y^r)$, otherwise $f_j=(-1)^{a+1}\bigl(\begin{smallmatrix} r\\j\end{smallmatrix}\bigr) x^jy^{r-j}$. 
\end{lem}
\begin{proof} It is enough to prove the statement when $a=0$, since twisting the action by $\det^a$ multiplies $f_j$ by $(\det s)^a=(-1)^a$.
 We have 
\begin{equation}\label{itgetskilled}
f_j=\sum_{\lambda\in \Fp} \lambda^{p-1-j} (\lambda x+y)^r= \sum_{i=0}^r \begin{pmatrix} r\\ i\end{pmatrix} (\sum_{\lambda\in \Fp}
\lambda^{p-1+i-j}) x^i y^{r-i}.
\end{equation}
If $a\ge 0$ is an integer then $\Lambda_a:=\sum_{\lambda\in \Fp} \lambda^{a}$ is zero, unless $a>0$ and $p-1$ divides $a$, 
in which case $\Lambda_a=-1$. Note that $0^0=1$. If $a=p-1+i-j$ then $\Lambda_a \neq 0$ if and only if $i=j$ or $i-j=p-1$,
which is equivalent to $r=i=p-1$ and $j=0$. This implies the assertion.
\end{proof}
Let $\Fbar[[I\cap U]]$ denote the completed group algebra of $I\cap U$. Since $I\cap U\cong \Zp$ mapping $X$ to 
$\begin{pmatrix} 1 & 1 \\ 0& 1\end{pmatrix} -1$ induces an isomorphism between the ring   of formal power series 
in one variable $\Fbar[[X]]$ and $\Fbar[[I\cap U]]$. Every smooth representation $\tau$ of $I\cap U$ is naturally a module over
 $\Fbar[[I\cap U]]$, and we will also view $\tau$ as a module over $\Fbar[[X]]$ via the above isomorphism.

\begin{lem}\label{calcsym2} Let $x^r\in \Sym^r\Fbar^2\otimes\det^a$ then $X^r s x^r= (-1)^a r! x^r$.
\end{lem}
\begin{proof} We have $s x^r=(-1)^a y^r$. If $0\le i\le r$ then $X \centerdot x^{r-i}y^i= x^{r-i} (y+1)^i- x^{r-i} y^i= i x^{r-i+1}y^{i-1} +Q$,
where $Q$ is a homogeneous polynomial of degree $r$, such that the degree of $Q$ in $y$ is less than $i-1$. Applying this observation $r$
times we obtain that $X^r \centerdot y^r= r! x^r.$
\end{proof}

\section{Irreducible representations of $G$}
We recall some facts about the irreducible representations of $G$. We fix an integer $r$ with $0\le r\le p-1$. We consider 
$\Sym^r\Fbar^2$ as a representation of $KZ$ by making $p$ act trivially. 
 It is shown in \cite[Prop. 8]{bl} that there exists an isomorphism of algebras:
$$\End_{G}(\cIndu{KZ}{G}{\Sym^r \Fbar^2})\cong \Fbar[T]$$
for a certain $T\in  \End_{G}(\cIndu{KZ}{G}{\Sym^r \Fbar^2})$ defined in \cite[\S 3]{bl}.  One may describe $T$ as follows. Let 
$\varphi\in \cIndu{KZ}{G}{\Sym^r \Fbar^2}$ be such that 
 $\supp \varphi =ZK$ and $\varphi(1)=x^r$. Since $\varphi$ generates $\cIndu{KZ}{G}{\Sym^r \Fbar^2}$ as 
a $G$-representation $T$ is determined by $T\varphi$. 
\begin{lem}\label{T}
 \begin{itemize}
\item[(i)] If $r=0$  then
$$ T \varphi = \Pi \varphi +\sum_{\lambda \in \Fp} \begin{pmatrix} 1 & [\lambda]\\ 0 & 1
\end{pmatrix} t \varphi .$$
\item[(ii)] Otherwise, 
$$ T \varphi = \sum_{\lambda \in \Fp} \begin{pmatrix} 1 & [\lambda] \\ 0 & 1
\end{pmatrix} t  \varphi.$$
\end{itemize}
\end{lem}
\begin{proof} In the notation of \cite{bl} this is a calculation of $T([1, e_{\vec{0}}])$. The claim follows from the formula (19) in the proof 
of \cite{bl} Theorem 19. 
 \end{proof}
 It follows from  \cite[Thm 19]{bl} that the map $T-\lambda$ is injective, for all $\lambda\in \Fbar$.
\begin{defi} Let $\pi(r,\lambda)$ be a representation of $G$ defined by the exact sequence:
\begin{displaymath}
\xymatrix{ 0 \ar[r] & \cIndu{ZK}{G}{\Sym^r \Fbar^2}\ar[r]^-{T-\lambda}&  \cIndu{ZK}{G}{\Sym^r \Fbar^2}\ar[r]& \pi(r,\lambda)\ar[r]& 0.}
\end{displaymath}
If $\eta: \Qp^{\times}\rightarrow \Fbar^{\times}$ is a smooth character then set $\pi(r,\lambda,\eta):=\pi(r,\lambda)\otimes \eta\circ \det$.
\end{defi}
It follows from \cite[Thm.30]{bl} and \cite[Thm.1.1]{b1} that $\pi(r,\lambda)$ is irreducible unless $(r,\lambda)=(0,\pm1)$ or $(r,\lambda)=(p-1,\pm1)$.
Moreover, one has non-split exact sequences:
\begin{equation}\label{zero}
 0 \rightarrow \mu_{\pm 1}\circ \det\rightarrow \pi(p-1,\pm 1)\rightarrow \Sp\otimes\mu_{\pm 1}\circ \det \rightarrow 0, 
\end{equation}
\begin{equation}\label{p-1}
 0 \rightarrow \Sp\otimes\mu_{\pm 1}\circ \det\rightarrow \pi(0,\pm 1)\rightarrow \mu_{\pm 1}\circ \det \rightarrow 0, 
\end{equation}
where $\Sp$ is the Steinberg representation of $G$, (we take \eqref{zero} as definition) and if $\lambda\in \Fbar^{\times}$ then 
$\mu_{\lambda}:\Qp^{\times} \rightarrow \Fbar^{\times}$, $x\mapsto \lambda^{\val(x)}$.
Further, if $\lambda\neq 0$ and $(r,\lambda)\neq (0,\pm 1)$ then \cite[Thm.30]{bl} asserts that 
\begin{equation}\label{induced}
 \pi(r,\lambda)\cong \Indu{P}{G}{\mu_{\lambda^{-1}}\otimes \mu_{\lambda}\omega^r}.
\end{equation}
It follows from \cite[Thm. 33]{bl} and \cite[Thm 1.1]{b1} that the irreducible smooth representations of $G$ with the central character fall 
into $4$ disjoint classes:
\begin{itemize}
\item[(i)] characters, $\eta\circ \det$;
\item[(ii)] special series, $\Sp\otimes\eta\circ \det$;
\item[(iii)] (irreducible) principal series $\pi(r, \lambda, \eta)$, $0< r\le p-1$, $\lambda\neq 0$, $(r,\lambda)\neq (p-1, \pm 1)$;
\item[(iv)] supersingular $\pi(r, 0, \eta)$, $0\le r\le p-1$.
\end{itemize} 

\subsection{Supersingular representations}\label{supersingularrepresentations}
We discuss the supersingular representations. Breuil has shown  \cite[Thm.1.1]{b1} that the representations $\pi(r,0,\eta)$ are irreducible and using the results \cite{bl} 
classified smooth irreducible representations of $G$ with a central character. 

\begin{defi}\label{supersingularrep} An irreducible representation $\pi$ with a central character is supersingular if 
$\pi\cong \pi(r,0,\eta)$ for some $0\le r \le p-1$ and a smooth character $\eta$.
\end{defi}

All the isomorphism between supersingular representations cor\-res\-pon\-ding to different $r$ and $\eta$ are given by 
\begin{equation}\label{intertwine}
\pi(r,0,\eta)\cong \pi(r,0,\eta\mu_{-1})\cong \pi(p-1-r,0,\eta\omega^{r})\cong \pi(p-1-r,0,\eta\omega^{r}\mu_{-1})
\end{equation}
see \cite[Thm. 1.3]{b1}. It follows from \cite[Cor.36]{bl} that an irreducible smooth representation of $G$ with a central character 
is supersingular if and only if it is not a subquotient of any principal series representation.

We fix a supersingular representation $\pi$ of $G$ and we are interested in $\Ext^1_G(\tau,\pi)$, where $\tau$ is an irreducible 
smooth representation of $G$. If $\eta:G\rightarrow \Fbar^{\times}$ is a smooth character, then twisting by $\eta$ induces an isomorphism
$$\Ext^1_G(\tau,\pi)\cong \Ext^1_G(\tau\otimes\eta, \pi\otimes\eta).$$
Hence, we may assume that $p\in Z$ acts trivially on $\pi$, so that $\pi\cong \pi(r,0, \omega^a)$, for some $0\le r\le p-1$, 
and $0\le a<p-1$. It follows from \cite[Thm. 3.2.4, Cor. 4.1.4]{b1} that $\pi^{I_1}$ is $2$-dimensional. Moreover, \cite[Cor. 4.1.5]{b1}
implies that there exists a basis $\{v_{\sigma}, v_{\tilde{\sigma}}\}$ of $\pi^{I_1}$, such that $\Pi v_{\sigma}=v_{\tsigma}$, 
$\Pi v_{\tsigma}=v_{\sigma}$ and there exists an isomorphism of $K$-representations: 
$$\langle K\centerdot v_{\sigma}\rangle \cong \sigma, \quad  \langle K\centerdot v_{\tsigma}\rangle \cong \tsigma,$$ 
where $\sigma:=\Sym^r\Fbar^2\otimes \det^a$. The group $H$ acts on $v_{\sigma}$ by a character $\chi$ and on $v_{\tsigma}$ by a character 
$\chi^s$. Explicitly,
\begin{equation}\label{explicitchi} 
  \chi(\begin{pmatrix} [\lambda] & 0\\0 & [\mu]\end{pmatrix})=\lambda^r(\lambda \mu)^a, \quad \forall \lambda,\mu\in \Fp^{\times}.
\end{equation}
 
\begin{lem}\label{relations} The following relations hold:
\begin{equation}\label{rel}
v_{\sigma}=(-1)^{a+1}\sum_{\lambda\in \Fp} \lambda^{p-1-r} \begin{pmatrix} 1 & [\lambda] \\ 0 & 1\end{pmatrix} t v_{\tsigma}; 
\end{equation}
\begin{equation}\label{trel} 
v_{\tsigma}=(-1)^{r+a+1}\sum_{\lambda\in \Fp} \lambda^{r} \begin{pmatrix} 1 & [\lambda] \\ 0 & 1\end{pmatrix} t v_{\sigma};
\end{equation}
\begin{equation}\label{relX}
X^r t v_{\tsigma}= (-1)^a r! v_{\sigma}, 
\quad X^{p-1-r}tv_{\sigma}=(-1)^{r+a}(p-1-r)! v_{\tsigma}.
\end{equation}
\end{lem}
\begin{proof} Since $tv_{\tsigma}= s \Pi v_{\tsigma}= s v_{\sigma}$ this is a calculation in $\Sym^r\Fbar^2\otimes\det^a$, which is done in 
Lemmas \ref{calcsym} and \ref{calcsym2}.
\end{proof}

\begin{defi}\label{M} $M:=\biggl \langle \biggl (\begin{smallmatrix} p^{\NN} & \Zp\\ 0 & 1\end{smallmatrix} \biggr ) \pi^{I_1}
\biggr \rangle$, $M_{\sigma}:= \biggl \langle \biggl ( \begin{smallmatrix} p^{2\NN} & \Zp\\ 0 & 1\end{smallmatrix} \biggr ) v_{\sigma}
\biggr \rangle$ and   
 $M_{\tsigma}:= \biggl \langle \biggl ( \begin{smallmatrix} p^{2\NN} & \Zp\\ 0 & 1\end{smallmatrix} \biggl )v_{\tsigma}\biggr \rangle.$
\end{defi}

\begin{lem}\label{Istable} The subspaces  $M$, $M_{\sigma}$, $M_{\tsigma}$ are stable under the action of $I$.
\end{lem}
\begin{proof} We prove the statement for $M$, the rest is identical. By definition $M$ is stable under $I\cap U$. Since 
$I=(I\cap P^s) (I\cap U)$ it is enough to show that $M$ is stable under $I\cap P^s$. Suppose that $g_1\in I\cap P^s$, $g_2\in I\cap U$. 
Since $I=(I\cap U)(I\cap P^s)$ there exists $h_2\in I\cap U$ and $h_1\in I\cap P^s$ such that  $g_1g_2=h_2 h_1$. Moreover, for  $n\ge 0$ 
we have $t^{-n} (I\cap P^s) t^n \subset I$. Hence, if $v\in \pi^{I_1}$ then $(t^{-n}h_1 t^n)v\in \pi^{I_1}$ and so 
$$ g_1 ( g_2 t^n v)= h_2 h_1 t^n v= h_2 t^n (t^{-n} h_1 t^n) v \in M, \quad \forall v\in \pi^{I_1} .$$
This implies that $M$ is stable under $I\cap P^s$.
\end{proof}

The isomorphism $\pi(r,0, \omega^a)\cong \pi(p-r-1, 0, \omega^{r+a})$ allows to exploit the symmetry between $M_{\sigma}$ and $M_{\tsigma}$. 
In particular, if we prove a statement about $M_{\sigma}$ which holds
for all $\sigma$, then it also holds for $M_{\tsigma}$ (with $\sigma$ replaced by $\tsigma$).  

\begin{prop}\label{Minj} The triples $\chi\hookrightarrow M_{\sigma}$ and $\chi^s\hookrightarrow M_{\tilde{\sigma}}$ are injective envelopes 
of $\chi$ and $\chi^s$ 
in $\Rep_{H(I_1\cap U)}$. In particular, $M_{\sigma}^{I_1\cap U}=\Fbar v_{\sigma}$ and  $M_{\tsigma}^{I_1\cap U}=\Fbar v_{\tsigma}$.
\end{prop}
\begin{proof} We will show the claim for $M_{\sigma}$. The relations \eqref{relX} imply that 
$$v_{\sigma}= (-1)^r r! (p-1-r)! X^{r+p(p-1-r)} t^2 v_{\sigma}.$$
For $n\ge 0$ define  $\lambda_n:= ((-1)^r r! (p-1-r)!)^n$, $e_0:=0$ and $e_n:= r+p(p-1-r)+ p^2 e_{n-1}$. Further define 
$M_{\sigma, n}:=\langle (I_1\cap U) t^{2n}v_{\sigma}\rangle$. Since $t^{2n}v_{\sigma}= \lambda_1 X^{p^{2n}e_1}t^{2(n+1)}v_{\sigma}$, 
$M_{\sigma, n}$ is contained in $M_{\sigma, n+1}$ and hence 
$$M_{\sigma}=\underset{\underset{n}{\longrightarrow}}{\lim} \ M_{\sigma, n}.$$
Since $v_{\sigma}= \lambda_n X^{e_n} t^{2n} v_{\sigma}$ 
and $Xv_{\sigma}=0$ we obtain an isomorphism $M_{\sigma,n}\cong \Fbar[X]/(X^{e_n+1})$. In particular, for all $n\ge 0$ we have 
$M_{\sigma,n}^{I_1\cap U}= \Fbar v_{\sigma}$, and 
so $M_{\sigma}^{I_1\cap U}=\Fbar v_{\sigma}$. Given $m\ge 0$, set $\mathcal U_m:=
\bigl (\begin{smallmatrix} 1 & \pF^m \\ 0 & 1\end{smallmatrix}\bigr )$, 
choose $n$ such that $e_n>p^m$ and define $M'_{\sigma,m}:= \langle (I_1\cap U)\centerdot X^{e_n+1-p^m} t^{2n}v_{\sigma}\rangle$.
Then $M'_{\sigma, m}\cong \Fbar[X]/(X^{p^m})\cong M_{\sigma}^{\mathcal U_m} $ is an injective envelope of $\chi$ in 
$\Rep_{H(I_1\cap U)/\mathcal U_m}$. Since
$M_{\sigma}=\underset{\longrightarrow}{\lim} \ M'_{\sigma,m}$ we obtain that $M_{\sigma}$ is an injective envelope of $\chi$ in 
$\Rep_{H(I_1\cap U)}$.
\end{proof}

\begin{lem}\label{tv_1} Let $n\ge 0$ be an odd integer then $t^n v_{\sigma}\in M_{\tsigma}$ and $t^n v_{\tsigma}\in M_{\sigma}$. Hence, 
$t M_\sigma\subset M_{\tsigma}$ and $tM_{\tsigma}\subset M_{\sigma}$. 
\end{lem}
\begin{proof} It follows from the definition that $t^2 M_{\tsigma}\subset M_{\tsigma}$. Hence, it is enough to consider $n=1$. Applying  
$t$ to \eqref{relX} we obtain $t v_{\sigma}= (-1)^a (r!)^{-1} X^{pr} t^2 v_{\tsigma} \in M_{\tsigma}$. 
If $k,m\ge 0$ are integers  and $m$ even then we have $t(X^k t^m v_{\sigma})= X^{p k} t^{m} (t v_{\sigma})$ and since $tv_{\sigma}\in M_{\tsigma}$ and 
$m$ is even we 
obtain $t(X^k t^m v_{\sigma})\in M_{\tsigma}$. The set $\{X^k t^m v_{\sigma}: k, m\ge 0, 2\mid m\}$ spans $M_{\sigma}$ as an $\Fbar$-vector 
space. Hence, $tM_{\sigma} \subset M_{\tsigma}$. The rest follows by symmetry. 
\end{proof}

\begin{lem}\label{sv} We have $sv_{\sigma}\in M_{\sigma}$ and $ s v_{\tsigma}\in M_{\tsigma}$.
\end{lem}
\begin{proof} Since $sv_{\sigma}= s \Pi v_{\tsigma}= tv_{\tsigma}$ this follows from Lemma \ref{tv_1}.
\end{proof}

\begin{lem}\label{MI} $M$ is the direct sum of its $I$-submodules $M_{\sigma}$ and $M_{\tsigma}$.
\end{lem} 
\begin{proof} Proposition \ref{Minj} implies that $(M_{\sigma}\cap M_{\tsigma})^{I_1}=M_{\sigma}^{I_1}\cap M_{\tsigma}^{I_1}=\Fbar v_{\sigma}\cap 
\Fbar v_{\tsigma} =0$.
Hence $M_{\sigma}\cap M_{\tsigma}=0$ and so it is enough to show that $M=M_{\sigma}+ M_{\tsigma}$. Clearly, $M_{\sigma}\subset M$ and 
$M_{\tsigma}\subset M$. 
Lemma \ref{tv_1} implies $M\subseteq M_{\sigma}+M_{\tsigma}$.
\end{proof} 
  
\begin{defi} We set $\pi_{\sigma}:=M_{\sigma}+\Pi\centerdot M_{\tsigma}$ and $\pi_{\tsigma}:=M_{\tsigma}+\Pi\centerdot M_{\sigma}$.
\end{defi}

\begin{prop}\label{G+} The subspaces $\pi_{\sigma}$ and $\pi_{\tsigma}$ are stable under the action of $G^+$.
\end{prop}
\begin{proof} We claim that $s\pi_{\sigma}\subseteq \pi_{\sigma}$. Now $s (\Pi M_{\tsigma})= t M_{\tsigma} \subset M_{\sigma}$ by 
Lemma \ref{tv_1}. It is enough to show that $s M_{\sigma} \subset \pi_{\sigma}$. By definition of $M_{\sigma}$ it is enough to show that
$s(ut^n v_{\sigma})\in \pi_{\sigma}$ for all $u\in I_1\cap U$ and all even non-negative integers $n$.
 Lemma \ref{sv} gives $s v_{\sigma}\in M_{\sigma}$ and 
if $n\ge 2$ is an even integer then $s t^n v_{\sigma} = \Pi t^{n-1}v_{\sigma} \in \Pi M_{\tsigma}$ by Lemma \ref{tv_1}. Since $s (K_1\cap U)s= I_1\cap U^s$
for all $u\in K_1\cap U$, and $n\ge 0$ even, we get that $ s u t^n v_{\sigma} \in \pi_{\sigma}$. If $u\in (I_1\cap U)\setminus (K_1\cap U)$ and 
$n>0$ even, then the matrix identity:  
\begin{equation}\label{trix}
\begin{pmatrix} 0 & 1 \\ 1 & 0\end{pmatrix} \begin{pmatrix} 1 & \beta \\ 0 & 1\end{pmatrix}=
\begin{pmatrix} -\beta^{-1} & 1 \\ 0 & \beta \end{pmatrix}\begin{pmatrix} 1 & 0 \\ \beta^{-1} & 1\end{pmatrix}
\end{equation}   
implies that $s u t^n v_{\sigma}\in M_{\sigma}$. This settles the claim. By symmetry $\pi_{\tsigma}$ is also stable under $s$, and 
since $\pi_{\sigma}=\Pi \pi_{\tsigma}$, we obtain that $\pi_{\sigma}$ is stable under $\Pi s \Pi^{-1}$. Lemma \ref{Istable} 
implies that $\pi_{\sigma}$ is stable under $I$. Since $s$, $\Pi s \Pi^{-1}$ and  $I$ generate $G^0$, we get that $\pi_{\sigma}$ is stable
under $G^0$. Since $Z$ acts by a central character, $\pi_{\sigma}$ is stable under $G^+=Z G^0$. The result for $\pi_{\tsigma}$ follows 
by symmetry.
\end{proof} 

\section{Extensions}
In this section we compute extensions of characters for different subgroups of $I$. 

\begin{defi} Let $\kappa^u$, $\varepsilon$, $\kappa^l : I_1\rightarrow \Fbar$ be functions defined as follows, 
for $A=\begin{pmatrix} a & b \\ c & d\end{pmatrix} \in I_1$ we set  
$$\kappa^u(A)=\omega(b), \quad \varepsilon(A)=\omega(p^{-1}(a-d)), \quad \kappa^l(A)=\omega(p^{-1}c),$$
where $\omega: \Zp\rightarrow \Fbar$ is the reduction map composed with the canonical embedding.
\end{defi}

\begin{prop}\label{homi} If $p\neq 2$ then $\Hom(I_1/Z_1, \Fbar)=\langle \kappa^u,\kappa^l\rangle$. If $p=2$ then 
$\dim \Hom(I_1/Z_1, \Fbar)=4$.
\end{prop} 
\begin{proof} Let $\psi: I_1/Z_1\rightarrow \Fbar$ be a continuous group homomorphism. Since $I_1\cap U\cong I_1 \cap U^s\cong \Zp$
there exist $\lambda, \mu\in \Fbar$ such that $\psi|_{I_1\cap U}=\lambda \kappa^u$ and $\psi|_{I_1\cap U^s}= \mu \kappa^l$. Then 
$\psi-\lambda \kappa^u -\mu\kappa^l$ is trivial on $I_1\cap U$ and $I_1\cap U^s$. The matrix identity 
\begin{equation}\label{ssap}
\begin{pmatrix} 1 & \beta \\ 0 & 1 \end{pmatrix} \begin{pmatrix} 1 & 0\\ \alpha & 1\end{pmatrix}= 
\begin{pmatrix} 1 & 0\\ \alpha(1+\alpha\beta)^{-1} & 1\end{pmatrix} \begin{pmatrix} (1+\alpha\beta) &  \beta \\
0 & (1+\alpha \beta)^{-1}\end{pmatrix}
\end{equation}
implies that $I_1\cap U$ and $I_1\cap U^s$ generate $I_1\cap \SL_2(\Qp)$. So $\psi- \lambda \kappa^u -\mu\kappa^l$ must factor through 
$\detr$. The image of $Z_1$ in $1+\pF$ under $\detr$ is $(1+\pF)^2$. If $p>2$ then $(1+\pF)^2=1+\pF$ and hence $\psi= \lambda \kappa^u +\mu\kappa^l$.
If $p=2$ then $\dim \Hom((1+\pF)/(1+\pF)^2, \Fbar)=2$.
\end{proof}

\begin{lem}\label{homip} Assume $p>2$ then  $\Hom((I_1\cap P)/Z_1, \Fbar)=\langle \kappa^u,\varepsilon\rangle$ and 
 $\Hom((I_1\cap P^s)/Z_1, \Fbar)=\langle \kappa^l,\varepsilon\rangle$.
\end{lem}
\begin{proof} Let $\psi: (I_1\cap P)/Z_1\rightarrow  \Fbar$ be a continuous group homomorphism. Since $I_1\cap U\cong \Zp$
there exist $\lambda\in \Fbar$ such that $\psi|_{I_1\cap U}=\lambda \kappa^u$. Then $\psi-\lambda \kappa^u$ is trivial on 
$I_1\cap U$, and hence defines a homomorphism $(I_1\cap P)/Z_1(I_1\cap U)\cong (T\cap I_1)/Z_1\rightarrow \Fbar$. Since 
$p>2$ we have an isomorphism $(T\cap I_1)/Z_1\cong 1+p\Zp\cong \Zp$. Hence, there 
exists $\mu\in \Fbar$ such that $\psi=\mu \varepsilon +\lambda \kappa^u$.
Conjugation by $\Pi$ gives the second assertion.
\end{proof}

\begin{prop}\label{extI} Let $\chi,\psi:H\rightarrow \Fbar^{\times}$ be characters. $\Ext^1_{I/Z_1}(\psi, \chi)$ is non-zero if and only
if  $\psi=\chi\alpha$ or $\psi=\chi\alpha^{-1}$. Moreover, 
\begin{itemize} 
 \item[(i)] if $p>3$ then $\dim \Ext^1_{I/Z_1}(\chi\alpha, \chi)=\dim \Ext^1_{I/Z_1}(\chi\alpha^{-1}, \chi)=1$;
 \item[(ii)] if $p=3$ then $\chi\alpha=\chi\alpha^{-1}$ and $\dim \Ext^1_{I/Z_1}(\chi\alpha, \chi)=2$;
 \item[(iii)] if $p=2$ then $\chi=\chi\alpha=\chi\alpha^{-1}=\Eins$ and $\dim \Ext^1_{I/Z_1}(\Eins, \Eins)=4$.
\end{itemize}  
\end{prop}
\begin{proof} Since the order of $H$ is prime to $p$ and $I=H I_1$ we have 
$$\Ext^1_{I/Z_1}(\psi,\chi)\cong \Hom_H(\psi, H^1(I_1/Z_1, \chi)).$$
Now $H^1(I_1/Z_1, \chi)\cong \Hom(I_1/Z_1, \Fbar)$,
where if $\xi\in \Hom(I_1/Z_1, \Fbar)$ and $h\in H$ then $[h\centerdot \xi](u)=\chi(h)\xi(h^{-1}uh)$. The assertion follows from 
Proposition \ref{homi}.
\end{proof}
  
Similarly one obtains:

\begin{lem}\label{up} Let $\chi, \psi:H\rightarrow \Fbar^{\times}$ be characters and let $\mathcal U=\left(\begin{smallmatrix} 1 & \pF^k\\ 0 & 1\end{smallmatrix}\right )$
for some integer $k$ then $\Ext^1_{H\mathcal U}(\psi, \chi)\neq 0$
if and only if  $\psi=\chi\alpha^{-1}$. Moreover,  $\dim \Ext^1_{H\mathcal U}(\chi\alpha^{-1}, \chi)=1$. 
\end{lem}

\begin{lem}\label{down} Let $\chi, \psi:H\rightarrow \Fbar^{\times}$ be characters and let 
$\mathcal U=\bigl(\begin{smallmatrix} 1 & 0 \\\pF^k & 1\end{smallmatrix}\bigr )$
for some integer $k$ then $\Ext^1_{H\mathcal U}(\psi, \chi)\neq 0$
if and only if  $\psi=\chi\alpha$. Moreover,  $\dim \Ext^1_{H\mathcal U}(\chi\alpha, \chi)=1$. 
\end{lem}

\begin{lem}\label{upt} Assume $p>2$ and let $\chi, \psi:H\rightarrow \Fbar^{\times}$ be characters  then $\Ext^1_{(I_1\cap P)/Z_1}(\psi, \chi)\neq 0$
if and only if  $\psi\in \{\chi, \chi\alpha^{-1}\}$. Moreover,  
$$\dim \Ext^1_{(I\cap P)/Z_1}(\chi\alpha^{-1}, \chi)= \dim \Ext^1_{(I\cap P)/Z_1}(\chi, \chi)=1.$$
\end{lem}

\begin{lem}\label{downt} Assume $p>2$ and let $\chi, \psi:H\rightarrow \Fbar^{\times}$ be characters  then $\Ext^1_{(I_1\cap P)/Z_1}(\psi, \chi)\neq 0$
if and only if  $\psi\in \{\chi, \chi\alpha\}$. Moreover,  
$$\dim \Ext^1_{(I\cap P^s)/Z_1}(\chi\alpha, \chi)= \dim \Ext^1_{(I\cap P^s)/Z_1}(\chi, \chi)=1.$$ 
\end{lem}

\begin{prop}\label{modX} Let $\chi:H\rightarrow \Fbar^{\times}$ be a character and let $\chi\hookrightarrow J_{\chi}$ be an injective envelope 
of $\chi$ in $\Rep_{H(I_1\cap U)}$, then $(J_{\chi}/\chi)^{I_1\cap U}$ is $1$-dimensional and $H$ acts on it by $\chi\alpha^{-1}$. Moreover,
$\chi\alpha^{-1}\hookrightarrow J_{\chi}/\chi$ is an injective envelope of $\chi\alpha^{-1}$ in $\Rep_{H(I_1\cap U)}$. 
\end{prop}
\begin{proof} Consider an exact sequence of $H(I\cap U)$-representations:
$$0\rightarrow \chi\rightarrow J_{\chi}\rightarrow J_{\chi}/\chi\rightarrow 0.$$
Since $J_{\chi}$ is an injective envelope of $\chi$ in $\Rep_{I\cap U}$ taking $I_1\cap U$ invariants induces $H$-equivariant
isomorphism $(J_{\chi}/\chi)^{I_1\cap U}\cong H^1(I_1\cap U, \chi)$. It follows from Lemma \ref{up} that 
$\dim (J_{\chi}/\chi)^{I_1\cap U}=1$ and $H$ acts on $(J_{\chi}/\chi)^{I_1\cap U}$ via the character $\chi\alpha^{-1}$. 
Let $J_{\chi\alpha^{-1}}$ be an injective envelope of $\chi\alpha^{-1}$ in $\Rep_{H(I_1\cap U)}$, then there  exists an exact sequence of 
$H(I_1\cap U)$-representations:
$$0\rightarrow J_{\chi}/\chi\rightarrow J_{\chi\alpha^{-1}}\rightarrow Q\rightarrow 0.$$
Since $J_{\chi\alpha^{-1}}$ is an essential extension of $\chi\alpha^{-1}$, we have $J_{\chi\alpha^{-1}}^{I_1\cap U}\cong \chi\alpha^{-1}$. 
Hence taking $(I_1\cap U)$-invariants induces an isomorphism $Q^{I_1\cap U}\cong H^1(I_1\cap U, J_{\chi}/\chi)\cong H^2(I_1\cap U, \chi)$. Since 
$I_1\cap U \cong \Zp$ is a free pro-$p$  group we have $H^2(I_1\cap U, \chi)=0$, see \cite[\S3.4]{cohgal}. Hence $Q^{I_1\cap U}=0$,
which implies $Q=0$.
\end{proof}

\begin{lem}\label{extriv} Let $\iota: J\hookrightarrow A$ be a monomorphism in an abelian category $\mathcal A$. If $J$ is an injective
object in $\mathcal A$ then there exists $\sigma: A\rightarrow J$ such that $\sigma \circ \iota=\id$.
\end{lem}
\begin{proof} Since $J$ is injective the map $\Hom_{\mathcal A}(A,J)\rightarrow \Hom_{\mathcal A}(J,J)$ is surjective.
\end{proof}

\section{Exact sequence}\label{exactsequence} 
Let $\pi:=\pi(r, 0, \eta)$ with $0\le r\le p-1$. We use the notation of \S\ref{supersingularrepresentations}, so that $\sigma:=\Sym^r\Fbar^2 \otimes \det^a$, with $\det^a=\eta\circ \det |_K$, and $\chi: H\rightarrow \Fbar^{\times}$ a character as in \eqref{explicitchi}. We construct an exact sequence of $I$-representations which will be used to calculate $H^1(I_1/Z_1, \pi)$.
 
\begin{lem}\label{calc1} If $r\neq 0$ then set 
$$w_{\sigma}:=\sum_{\lambda\in \Fp} \lambda^{p-r} \begin{pmatrix} 1 & [\lambda]\\ 0 & 1 \end{pmatrix} t v_{\tsigma}+ (\sum_{\mu\in \Fp} \mu) v_{\sigma}.$$
Then $w_{\sigma}$ is fixed by $I_1\cap P^s$ and 
$$\begin{pmatrix} 1 & 1 \\ 0 & 1\end{pmatrix} w_{\sigma}= w_{\sigma} - (-1)^a r v_{\sigma}.$$

If $r=0$ then set 
$$w_{\sigma}:= \sum_{\lambda, \mu\in \Fp}\lambda\begin{pmatrix} 1 & [\mu]+p[\lambda] \\ 0 & 1\end{pmatrix} t^2 v_{\sigma}.$$ 
Then 
$$\begin{pmatrix} 1 & 1\\ 0 & 1\end{pmatrix} w_{\sigma}=w_{\sigma} +v_{\sigma}, \quad \begin{pmatrix} 1 & 0\\ p & 1\end{pmatrix} w_{\sigma}=
w_{\sigma}-(\sum_{\mu\in \Fp}\mu^2)v_{\sigma}.$$
If $\alpha\in [x]+\pF$, $\beta\in [y]+\pF$ then
$$\begin{pmatrix} 1+p\alpha & 0\\ 0 & 1+p\beta\end{pmatrix} w_{\sigma}=w_{\sigma} +(x-y)(\sum_{\mu\in\Fp} \mu)v_{\sigma}.$$
\end{lem}

\begin{proof} We set $w:=w_{\sigma}$. Suppose that $r\neq 0$. Now  $tv_{\tsigma}= s\Pi v_{\tsigma}= sv_{\sigma}$. Hence, if we identify $v_{\sigma}$ with 
$x^r\in \Sym^r \Fbar^2\otimes \det^a$ then Lemma \ref{calcsym} applied to $j=r-1$ gives $w=-(-1)^a r x^{r-1}y$. This implies the assertion.  
 
Suppose that $r=0$ and let $P(X):=\frac{X^p+1-(X+1)^p}{p}\in \mathbb Z[X]$, then \cite{serre} implies that 
\begin{equation}
\begin{split}
\begin{pmatrix} 1 & 1 \\ 0 & 1 \end{pmatrix} w&= \sum_{\lambda,\mu\in \Fp} 
\lambda\begin{pmatrix} 1 & 1+[\mu]+p[\lambda] \\ 0 & 1\end{pmatrix} t^2 v_{\sigma}\\ &=  \sum_{\lambda,\mu\in \Fp} 
\lambda\begin{pmatrix} 1 & [\mu+1]+p[\lambda+P(\mu)] \\ 0 & 1\end{pmatrix} t^2 v_{\sigma}.
\end{split}
\end{equation}
Hence,   
\begin{equation} 
\begin{split}
\begin{pmatrix} 1 & 1 \\ 0 & 1 \end{pmatrix} w& =  
\sum_{\lambda,\mu\in \Fp} 
\lambda\begin{pmatrix} 1 & [\mu]+p[\lambda+P(\mu-1)] \\ 0 & 1\end{pmatrix} t^2 v_{\sigma}\\
&=\sum_{\lambda,\mu\in \Fp} 
(\lambda-P(\mu-1))\begin{pmatrix} 1 & [\mu]+p[\lambda] \\ 0 & 1\end{pmatrix} t^2 v_{\sigma}\\
&= w- \sum_{\lambda,\mu\in \Fp} P(\mu-1)\begin{pmatrix} 1 & [\mu]+p[\lambda] \\ 0 & 1\end{pmatrix} t^2 v_{\sigma}\\
&= w+(-1)^a\sum_{\mu\in\Fp} P(\mu-1)\begin{pmatrix} 1 & [\mu] \\ 0 & 1\end{pmatrix} t v_{\tsigma}\\ &= w+(\sum_{\mu\in\Fp} P(\mu-1))v_{\sigma},
 \end{split}
\end{equation}
where the last two equalities follow from  \eqref{rel}, \eqref{trel}. If $p=2$ then $P(X-1)=1-X$, otherwise 
 $P(X-1)=\sum_{i=1}^{p-1} p^{-1} \begin{pmatrix} p\\ i\end{pmatrix} X^i(-1)^{p-i}$. Hence  $\sum_{\mu\in \Fp} P(\mu-1)=
-\sum_{\mu\in \Fp^{\times}} \mu^{p-1}=1$. 

Now  $t^2v_{\sigma}$ is fixed by $\begin{pmatrix} 1 & \pF^2\\ 0 & 1\end{pmatrix}$  and $I_1\cap P^s$, so the matrix identity
\begin{equation}\label{pass}
\begin{pmatrix} 1 & 0 \\ \beta & 0 \end{pmatrix} \begin{pmatrix} 1 & \alpha\\ 0 & 1\end{pmatrix}= 
\begin{pmatrix} 1 & \alpha(1+\alpha\beta)^{-1} \\ 0 & 1\end{pmatrix} \begin{pmatrix} (1+\alpha\beta)^{-1} & 0 \\ \beta & 
1+\alpha \beta\end{pmatrix}
\end{equation}
implies that 
\begin{equation}
\begin{split}
\begin{pmatrix} 1 & 0 \\ p & 1\end{pmatrix} w&= \sum_{\lambda,\mu\in\Fp} \lambda 
\begin{pmatrix} 1 & [\mu]+ p[\lambda-\mu^2] \\ 0 & 1\end{pmatrix}t^2 v_{\sigma}\\ 
&=\sum_{\lambda,\mu\in \Fp} (\lambda+\mu^2) \begin{pmatrix} 1 & [\mu]+ p[\lambda] \\ 0 & 1\end{pmatrix}t^2 v_{\sigma}=w-(\sum_{\mu\in \Fp} \mu^2)v_{\sigma}.
\end{split}
\end{equation} 
If $\alpha\in [x] +\pF$ and $\beta\in [y]+\pF$ then the same argument gives 
\begin{equation}\label{pass2}
\begin{split}
\begin{pmatrix} 1+p\alpha & 0 \\ 0 & 1+p\beta\end{pmatrix} w&= \sum_{\lambda, \mu\in \Fp}
\lambda\begin{pmatrix} 1 & [\mu]+p[\lambda +\mu (x-y)] \\ 0 & 1\end{pmatrix} t^2 v_{\sigma}\\
&= w+(x-y)(\sum_{\mu\in \Fp} \mu)v_{\sigma}.
\end{split}
\end{equation}
\end{proof}

\begin{prop}\label{fix} We have $(M_{\sigma}/\Fbar v_{\sigma})^{I_1\cap U} =(M_{\sigma}/\Fbar v_{\sigma})^{I_1}$. Moreover, let 
$\Delta_{\sigma}$ be the image of  $(M_{\sigma}/\Fbar v_{\sigma})^{I_1}$ in $H^1(I_1, \Eins)\cong \Hom(I_1, \Fbar)$. Then the following hold:
 \begin{itemize} 
 \item[(i)] if either $r\neq 0$ or $p>3$ then $\Delta_{\sigma}=\Fbar \kappa^u$;
\item[(ii)] if $p=3$ and $r=0$ then $\Delta_{\sigma}=\Fbar(\kappa^u+\kappa^l)$;
\item[(iii)] if $p=2$ and $r=0$ then $\Delta_{\sigma}=\Fbar(\kappa^u+\kappa^l +\varepsilon)$.
\end{itemize}
\end{prop}  

\begin{proof} It follows from Proposition \ref{modX} that $(M_{\sigma}/\Fbar v_{\sigma})^{I_1\cap U}$ is $1$\--di\-men\-sio\-nal. Since  
$(M_{\sigma}/\Fbar v_{\sigma})^{I_1}\neq 0$
the  inclusion $(M_{\sigma}/\Fbar v_{\sigma})^{I_1}\subseteq (M_{\sigma}/\Fbar v_{\sigma})^{I_1\cap U}$ is an equality. The image of $w_{\sigma}$ of 
Lemma \ref{calc1} spans $(M_{\sigma}/\Fbar v_{\sigma})^{I_1}$ and the last assertion follows from Lemma \ref{calc1}.

\end{proof}

\begin{thm}\label{exseq} The map $(v,w)\mapsto v-w$ induces an exact sequence of $I$-representations:
$$0\rightarrow \pi^{I_1}\rightarrow M\oplus \Pi\centerdot M\rightarrow \pi\rightarrow 0.$$ 
\end{thm}

\begin{proof} We claim that $M\cap \Pi\centerdot M=\pi^{I_1}$. Consider an exact sequence:
$$0\rightarrow \pi^{I_1}\rightarrow M\cap \Pi\centerdot M\rightarrow Q\rightarrow 0.$$
Since $M\cap \Pi\centerdot M$ is an $I_1$-invariant subspace of $\pi$, we have $(M\cap \Pi\centerdot M)^{I_1}\subseteq \pi^{I_1}$. 
Since $M\cap \Pi\centerdot M$ contains $\pi^{I_1}$ the inclusion is an equality. Hence, by taking $I_1$-invariants we obtain an 
injection $\partial: Q^{I_1}\hookrightarrow H^1(I_1, \pi^{I_1})\cong \Hom(I_1, \Fbar)\oplus \Hom(I_1, \Fbar)$.
The element $\Pi$ acts on $H^1(I_1, \pi^{I_1})$ by $\Pi\centerdot (\psi_1, \psi_2)= (\psi_2^{\Pi}, \psi_1^{\Pi})$.
Let $\Delta_{\sigma}$ (resp. $\Delta_{\tsigma}$) denote the image of $(M_{\sigma}/\Fbar v_{\sigma})^{I_1}$ (resp. $(M_{\tsigma}/\Fbar v_{\tsigma})^{I_1}$) in 
$\Hom(I_1, \Fbar)$. 
Let $\Delta$ be the  image of $(M/\pi^{I_1})^{I_1}$ in $H^1(I_1, \pi^{I_1})$ so that  $\Delta=\Delta_{\sigma}\oplus \Delta_{\tsigma}$.
By taking $I_1$-invariants of the diagram 
\begin{displaymath} 
\xymatrix@1{ 0\ar[r] & \pi^{I_1}\ar[r]\ar[d]& M\cap\Pi\centerdot M \ar[r]\ar[d] & Q\ar[r]\ar[d] & 0\\
           0\ar[r] & \pi^{I_1}\ar[r]& M \ar[r] & M/\pi^{I_1}\ar[r] & 0}
\end{displaymath}
we obtain a commutative diagram:
\begin{displaymath}
\xymatrix@1{ \;Q^{I_1}\;\ar[d]\ar@{^(->}[r]^-{\partial} & H^1(I_1, \pi^{I_1})\ar[d]^{\id}\\
             \;(M/\pi^{I_1})^{I_1}\; \ar@{^{(}->}[r]^-{\partial} &  H^1(I_1, \pi^{I_1}).}
\end{displaymath}
and hence an injection $\partial(Q^{I_1})\hookrightarrow \Delta$. Acting by $\Pi$ we obtain an injection
$\partial(Q^{I_1})\hookrightarrow \Pi\centerdot \Delta$. We claim that $\Delta\cap \Pi\centerdot \Delta=0$. We have
$$\Delta\cap \Pi\centerdot \Delta=(\Delta_{\sigma}\cap \Pi\centerdot \Delta_{\tsigma} )\oplus  
(\Delta_{\tsigma}\cap \Pi\centerdot \Delta_{\sigma}).$$ 
By symmetry we may assume $r<p-1$. Proposition \ref{fix} applied to $M_{\sigma}$ and $M_{\tsigma}$ implies that 
 that  if $r\neq 0$ then $\Delta= \Fbar \kappa^u\oplus \Fbar \kappa^u$, hence $\Pi\centerdot \Delta= \Fbar(\kappa^u)^{\Pi}  \oplus  \Fbar (\kappa^u)^{\Pi}=
\Fbar\kappa^l\oplus \Fbar\kappa^l$, so that $\Delta\cap \Pi\centerdot \Delta=0$. If $r=0$ then Proposition \ref{fix} implies that 
$\Delta= \Fbar (\kappa^u-(\sum_{\mu\in\Fp} \mu^2) \kappa^l +(\sum_{\mu\in \Fp} \mu )\varepsilon)\oplus \Fbar \kappa^u$, hence 
$\Pi\centerdot \Delta=\Fbar \kappa^l \oplus  \Fbar (\kappa^l-(\sum_{\mu\in\Fp} \mu^2) \kappa^u -(\sum_{\mu\in \Fp} \mu )\varepsilon)$, again 
$\Delta\cap \Pi\centerdot \Delta=0$. Note that if $r=0$ then we have to apply Proposition \ref{fix} to $M_{\tsigma}$ with $r=p-1$, and $p-1\neq 0$.
This implies that $Q^{I_1}=0$ and hence $Q=0$.

Since $G^+$ and $\Pi$ generate $G$, Proposition \ref{G+} implies that $\pi_{\sigma}+\pi_{\tsigma}$ is stable under the action of $G$. 
Since $\pi$ is irreducible we get $\pi=\pi_{\sigma}+\pi_{\tsigma}$. This implies surjectivity.  
\end{proof}

\begin{cor}\label{one} We have $M_{\sigma}\cap \Pi\centerdot M_{\tsigma}=\pi_{\sigma}^{I_1}=\Fbar v_{\sigma}$ and 
$M_{\tsigma}\cap \Pi\centerdot M_{\sigma}=\pi_{\tsigma}^{I_1}=\Fbar v_{\tsigma}$.
\end{cor}
\begin{proof} It is enough to show that $\pi_{\sigma}^{I_1}=\Fbar v_{\sigma}$, since by Theorem \ref{exseq} 
$M_{\sigma}\cap \Pi\centerdot M_{\tsigma}$ is contained in $\pi^{I_1}$.
Suppose not. Clearly $v_{\sigma}\in \pi_{\sigma}$, so since $\pi^{I_1}$ is $2$-dimensional, we obtain that $v_{\tsigma}\in \pi_{\sigma}$. Then there
exists $u_1\in M_{\sigma}$ and $u_2\in \Pi\centerdot M_{\tsigma}$ such that $v_{\tsigma}=u_1+u_2$. So $u_2\in \Pi\centerdot M_{\tsigma}\cap (M_{\sigma}+M_{\tsigma})\subset 
\pi^{I_1}$ by Theorem \ref{exseq}. Hence $u_2=\lambda v_{\sigma}$ for some $\lambda\in \Fbar$, and so $u_2\in M_{\sigma}$, and so $v_{\tsigma}\in M_{\sigma}$.
This contradicts $M_{\sigma}\cap M_{\tsigma}=0$. 
\end{proof} 

\begin{cor}\label{two} As $G^+$-representation $\pi$ is the direct sum of its subrepresentations $\pi_{\sigma}$ and $\pi_{\tsigma}$.
\end{cor}
\begin{proof} It follows from Theorem \ref{exseq} that $\pi=\pi_{\sigma}+\pi_{\tsigma}$. Now 
$$(\pi_{\sigma}\cap \pi_{\tsigma})^{I_1}=\pi_{\sigma}^{I_1}\cap \pi_{\tsigma}^{I_1}=\Fbar v_{\sigma} \cap \Fbar v_{\tsigma}=0.$$
Hence, $\pi_{\sigma}\cap \pi_{\tsigma}=0$.
\end{proof}

\begin{cor}\label{three} We have $\pi\cong \Indu{G^+}{G}{\pi_{\sigma}}\cong  \Indu{G^+}{G}{\pi_{\tsigma}}$.
\end{cor}

\section{Computing $H^1(I_1/Z_1, \pi)$}

We keep the notation of \S\ref{exactsequence} and compute $H^1(I_1/Z_1, \pi)$ as a representation of $H$ under the assumption $p>2$.

\begin{lem}\label{extXP}  Assume that $p>2$. Let $\psi,\chi: H\rightarrow \Fbar^{\times}$ be characters. Let $N$ be a smooth 
representation of $(I\cap P)/Z_1$, such that $N|_{H(I_1\cap U)}$ is an injective envelope of $\chi$ in $\Rep_{H(I_1\cap U)}$. Suppose that 
$\Ext^1_{(I\cap P)/Z_1}(\psi, N)\neq 0$ then $\psi=\chi$. Moreover, $\Ext^1_{(I\cap P)/Z_1}(\chi, N)\cong \Ext^1_{(I\cap P)/Z_1}(\chi,\chi)$ 
is $1$-dimensional.
\end{lem}
\begin{proof} Suppose that we have a non-split extension $0\rightarrow N\rightarrow E\rightarrow \psi\rightarrow 0$. Since 
$N|_{H(I_1\cap U)}$ is injective Lemma \ref{extriv} implies that the extension splits when restricted to $H(I_1\cap U)$. Hence, 
there exists $v\in E^{I_1\cap U}$ such that $H$ acts on $v$ by $\psi$ and the image of $v$ 
spans the underlying vector space of $\psi$. If $v$ is fixed by $I_1\cap T$, then since $I_1\cap T$ and $H(I_1\cap U)$ generate $I\cap P$ 
we would obtain a splitting of $E$ as an  $I\cap P$-representation. Hence, there exists some $h\in I_1\cap T$, such that $(h-1)v\in N$ is non-zero.
Since $h$ normalizes $I_1\cap U$ and $v$ is fixed by $I_1\cap U$, we obtain that $(h-1)v \in N^{I_1\cap U}$. 
Since $H$ acts on $v$ by $\psi$ and $T$ is abelian,  we get that $H$ acts on $(h-1)v$ by $\psi$. Since $N|_{H(I_1\cap U)}$ is an injective envelope of $\chi$
we obtain that $\chi=\psi$. 

By Proposition \ref{modX}, $N/\chi$ is an injective envelope of $\chi\alpha^{-1}$. Since $p>2$, $\chi\neq 
\chi\alpha^{-1}$ and so $\Hom_{I\cap P}(\chi, N/\chi)= \Ext^1_{(I\cap P)/Z_1}(\chi, N/\chi)=0$. So applying $\Hom_{I\cap P}(\chi, \centerdot)$ to the short exact sequence 
of $(I\cap P)/Z_1$ representations $0\rightarrow \chi\rightarrow N\rightarrow N/\chi\rightarrow 0$ gives us an isomorphism 
$\Ext^1_{(I\cap P)/Z_1}(\chi, N)\cong \Ext^1_{(I\cap P)/Z_1}(\chi,\chi)$. Lemma \ref{upt} implies that these spaces are $1$-dimensional.
\end{proof}
 
\begin{prop}\label{extX} Assume that $p>2$. Let $\psi,\chi: H\rightarrow \Fbar^{\times}$ be characters. Let $N$ be a smooth 
representation of $I/Z_1$, such that $N|_{H(I_1\cap U)}$ is an injective envelope of $\chi$ in $\Rep_{H(I_1\cap U)}$. Suppose that 
$\Ext^1_{I/Z_1}(\psi, N)\neq 0$ and let $\mathcal K$ be the kernel of the restriction map $\Ext^1_{I/Z_1}(\psi, N)\rightarrow  
\Ext^1_{(I\cap P)/Z_1}(\psi, N)$ then one of the following holds:
\begin{itemize}
\item[(i)] if $\mathcal K\neq 0$ then $\psi=\chi\alpha$;
\item[(ii)] if $\mathcal K=0$ then $\psi=\chi$. 
\end{itemize} 
Moreover, $\dim\Ext^1_{I/Z_1}(\chi\alpha, N)=1,$ and let $R$ be the submodule of $N$, fitting in the exact sequence
$0\rightarrow N^{I_1}\rightarrow R\rightarrow (N/N^{I_1})^{I_1}\rightarrow 0$, then there exists an exact sequence:
$$0\rightarrow \Hom_I(\chi,\chi\alpha^{-2})\rightarrow \Ext^1_{I/Z_1}(\chi,R)\rightarrow \Ext^1_{I/Z_1}(\chi,N)\rightarrow 0.$$
\end{prop}
\begin{proof} Suppose that $\mathcal K\neq 0$ then there exists a non-split extension $0\rightarrow N\rightarrow E\rightarrow \psi\rightarrow 0$
of $I/Z_1$-representations, which splits when restricted to $I\cap P$. Hence, there exists $v\in E^{I_1\cap P}$ such that $H$ acts on $v$ by $\psi$ and the image of $v$ 
spans the underlying vector space of $\psi$.  Let $k$ be the smallest integer $k\ge 1$ such that $v$ is fixed by $
\bigl (\begin{smallmatrix} 1 & 0 \\ \pF^k & 1 \end{smallmatrix}\bigr )$. If $k=1$ then $v$ is fixed by $I\cap U^s$. Since  $I\cap U^s$ and $I\cap P$
generate $I$, we would obtain that $I$ acts on $v$ by $\psi$ and hence the extension splits. Hence, $k$ is at least $2$. Set 
$\mathcal U:=\bigl ( \begin{smallmatrix} 1 & 0\\ \pF^{k-1} & 1 \end{smallmatrix}\bigr )$.
Our assumption on $k$ implies that $v':=\bigl ( \begin{smallmatrix} 1 & 0 \\ p^{k-1} & 1 \end{smallmatrix} \bigr )v-v\in N$ is non-zero.
The matrix identity \eqref{ssap} implies that $v'$ is fixed by $I_1\cap U$. Since  $N^{I_1\cap U}$ is $1$-dimensional and $H$ acts on $N^{I_1\cap U}$
by $\chi$, we obtain a non-zero element in $\Ext^1_{H\mathcal U}(\psi, \chi)$. Lemma \ref{down} implies that  
$\psi= \chi\alpha$. Let $\bar{v}$ be the image of $v$ in $E/N^{I_1}$. Again  by Proposition \ref{modX} 
$(N/N^{I_1})^{I_1\cap U}$ is $1$-dimensional and $H$ acts on $(N/N^{I_1})^{I_1\cap U}$
by $\chi\alpha^{-1}$. If the extension $0\rightarrow N/N^{I_1}\rightarrow E/N^{I_1}\rightarrow \psi\rightarrow 0$ is non-split,
then by the same argument we would obtain a non-zero element in $\Ext^1_{H\mathcal U'}(\chi\alpha,\chi\alpha^{-1})$, where 
$\mathcal U':=\bigl (\begin{smallmatrix} 1 & 0\\ \pF^{m} & 1 \end{smallmatrix}\bigr )$, for some $m\ge 1$. This contradicts Lemma \ref{down}, as $p>2$ and so 
$\alpha$ is non-trivial. Hence we obtain an exact sequence:
\begin{equation}\label{yetanotherseq}
0\rightarrow \Hom_I(\chi\alpha,\chi\alpha^{-1})\rightarrow \Ext^1_{I/Z_1}(\chi\alpha,\chi)\rightarrow \Ext^1_{I/Z_1}(\chi\alpha,N)\rightarrow 0.
\end{equation}   
If $p>3$ then $\dim \Hom_I(\chi\alpha,\chi\alpha^{-1})=0$ and $\dim  \Ext^1_{I/Z_1}(\chi\alpha,\chi)=1$. If $p=3$ then 
$\dim \Hom_I(\chi\alpha,\chi\alpha^{-1})=1$ and $\dim  \Ext^1_{I/Z_1}(\chi\alpha,\chi)=2$. Hence, $\dim \Ext^1_{I/Z_1} (\chi\alpha,N)=1$.

Assume that $\mathcal K=0$. Since we have assumed that $\Ext^1_{I/Z_1}(\psi,N)\neq 0$ we obtain that $\Ext^1_{(I\cap P)/Z_1}(\psi, N)\neq 0$ and 
Lemma \ref{extXP} implies that $\psi=\chi$ and $\dim \Ext^1_{I/Z_1}(\chi,N)\le 1$. Suppose that there exists a non-split extension 
$0\rightarrow N\rightarrow E\rightarrow \chi\rightarrow 0$
of $I/Z_1$-representations, which remains non-split when restricted to $I\cap P$. Let $w_1$ be a basis vector of $N^{I_1\cap U}$. Lemmas \ref{extXP}, \ref{upt}
and \ref{homip} imply that there exists $v\in E$ such that $H$ acts on $v$ by $\chi$ and for all $g\in I_1\cap P$ we have $g v= v+\varepsilon(g) w_1$.
In particular, $v$ is fixed by $I\cap U$ and  $\bigl (\begin{smallmatrix} 1 +\pF^2 & 0\\ 0 & 1+\pF^2\end{smallmatrix} \bigr )$.
As before, let $k$ be the smallest integer $k\ge 1$ such that $v$ is fixed by $
\bigl (\begin{smallmatrix} 1 & 0 \\ \pF^k & 1 \end{smallmatrix}\bigr )$. We claim that $k=2$. Indeed, if $k>2$ then let $v':=\bigl ( \begin{smallmatrix} 1 & 0 \\ p^{k-1} & 1 \end{smallmatrix}\bigr )v -v$. Then $v'\in N$ is non-zero, and 
the matrix identity \eqref{ssap} implies that $v'$ is fixed by $I_1\cap U$. Since  $N^{I_1\cap U}$ is $1$-dimensional and $H$ acts on 
$N^{I_1\cap U}$ by $\chi$, we obtain a non-zero element in $\Ext^1_{H\mathcal U}(\chi, \chi)$, with 
$\mathcal U:=\bigl ( \begin{smallmatrix} 1 & 0\\ \pF^{k-1} & 1 \end{smallmatrix}\bigr )$.
Lemma \ref{down} implies that  $\chi= \chi\alpha$. Since $p>2$ this cannot happen.

Consider $u:=\bigl ( \begin{smallmatrix} 1 & 0 \\ p & 1\end{smallmatrix} \bigr )v-v$. Using \eqref{ssap} and the fact that 
$k\ge 2$ we obtain
\begin{equation}
\begin{split}
\begin{pmatrix} 1 & 1 \\ 0 & 1 \end{pmatrix} u&=\begin{pmatrix} 1 & 0 \\ p(1+p)^{-1} & 1 \end{pmatrix} \begin{pmatrix} 1+p & 1 \\
0 & (1+p)^{-1}\end{pmatrix} v -v\\  &= \begin{pmatrix} 1 & 0 \\ p(1+p)^{-1} & 1 \end{pmatrix} (v+ 2w_1)-v= u+2w_1.
\end{split}
\end{equation}  
Since $2w_1\neq 0$ we get $u\neq 0$ and so $k=2$. By Proposition \ref{modX} $(N/\Fbar w_1)^{I_1\cap U}$ is $1$-dimensional. 
This implies that $(N/\Fbar w_1)^{I_1\cap U}\cong (N/\Fbar w_1)^{I_1}$ and 
the image of $u$ in $N/\Fbar w_1$ spans $(N/\Fbar w_1)^{I_1\cap U}$. If we set $R:=\langle w_1, u\rangle$ then by construction we obtain that 
the map $\Ext^1_{I/Z_1}(\chi, N)\rightarrow \Ext^1_{I/Z_1}(\chi,N/R)$ is zero. Proposition \ref{modX} implies that $(N/R)^{I_1}$ is $1$-dimensional 
and $H$ acts on it by a character $\chi\alpha^{-2}$. This implies the claim.

\end{proof}

\begin{cor}\label{ext2} Assume $p>2$ then the restriction maps 
$$\Ext^2_{I/Z_1}(\chi,\chi)\rightarrow \Ext^2_{(I\cap P^s)/Z_1}(\chi,\chi),$$ 
$$\Ext^2_{I/Z_1}(\chi,\chi)\rightarrow \Ext^2_{(I\cap P)/Z_1}(\chi,\chi)$$
are injective. 
\end{cor}
\begin{proof} Consider the exact sequence of $I$-representations $0\rightarrow \chi\rightarrow \Indu{I\cap P^s}{I}{\chi}\rightarrow Q\rightarrow 0$.
Iwahori decomposition implies that 
$$ (\Indu{I\cap P^s}{I}{\chi})|_{H(I_1\cap U)}\cong \Indu{H}{H(I_1\cap U)}{\chi},$$ and 
hence it is an injective envelope of $\chi$ in $\Rep_{H(I_1\cap U)}$. Proposition \ref{modX} implies that $Q|_{H(I_1\cap U)}$ is an injective 
envelope of $\chi\alpha^{-1}$ in $\Rep_{H(I_1\cap U)}$. Since $p>2$ Lemma \ref{extI} implies that $\Ext^1_{I/Z_1}(\chi,\chi)=0$, so
using Shapiro's lemma we obtain an exact sequence:
\begin{displaymath}
\begin{split}
\Ext^1_{(I\cap P^s)/Z_1}(\chi,\chi)\hookrightarrow \Ext^1_{I/Z_1}(\chi,Q)&\rightarrow \Ext^2_{I/Z_1}(\chi,\chi)\\
&\rightarrow\Ext^2_{(I\cap P^s)/Z_1}(\chi,\chi).
\end{split}
\end{displaymath}
Now $\dim  \Ext^1_{(I\cap P^s)/Z_1}(\chi,\chi)=1$ and $\dim  \Ext^1_{I/Z_1}(\chi,Q)=1$ by Proposition \ref{extX}. This implies the result
for $I\cap P^s$. By conjugating by $\Pi$ we obtain the result for $I\cap P$. 
\end{proof}

\begin{cor}\label{take} Assume $p>2$ and let $N$ be as in Proposition \ref{extX} then $\dim\Ext^1_{I/Z_1}(\chi,N)=1$, the natural 
maps 
\begin{equation}\label{natural1}
\Ext^2_{I/Z_1}(\chi,\chi)\rightarrow \Ext^2_{I/Z_1}(\chi, N),
\end{equation}
\begin{equation}\label{natural2}
\Ext^1_{I/Z_1}(\chi,N)\rightarrow \Ext^1_{(I\cap P)/Z_1}(\chi,N)
\end{equation}
 are injective and \eqref{natural2} is an isomorphism.
\end{cor} 
\begin{proof} We have an exact sequence:
$$\Ext^1_{I/Z_1}(\chi, N)\hookrightarrow \Ext^1_{I/Z_1}(\chi,N/\chi)\rightarrow \Ext^2_{I/Z_1}(\chi,\chi).$$ 
Proposition \ref{modX} and Lemma \ref{extXP} imply that  $\Ext^1_{(I\cap P)/Z_1}(\chi,N/\chi)=0.$ The commutative diagram:
\begin{displaymath}
\xymatrix@1{ \Ext^1_{I/Z_1}(\chi, N/\chi)\ar[r]\ar[d]^{0}&
\Ext^2_{I/Z_1}(\chi, \chi)\ar@{^(->}[d]^{\ref{ext2}}\\
\Ext^1_{(I\cap P)/Z_1}(\chi, N/\chi)\ar[r]^-{0}&
\Ext^2_{(I\cap P)/Z_1}(\chi, \chi)}
\end{displaymath}
and Corollary \ref{ext2} implies that $\Ext^1_{I/Z_1}(\chi,N/\chi)\rightarrow \Ext^2_{I/Z_1}(\chi,\chi)$ is the zero map.
Hence, \eqref{natural1} is injective and 
$$\dim \Ext^1_{I/Z_1}(\chi, N)=\dim  \Ext^1_{I/Z_1}(\chi,N/\chi)=1,$$
where the last equality is given by Propositions \ref{modX} and  \ref{extX}. We know that $\Ext^1_{I/Z_1}(\chi, N)\neq 0$.
So if \eqref{natural2} is not injective, then Proposition \ref{extX} gives $\chi=\chi\alpha$, but this cannot hold, since $p>2$.
Since both sides have dimension $1$, \eqref{natural2} is an isomorphism.
\end{proof}
\subsection{$p=3$}
The case $p=3$ requires some extra arguments. If you are only interested in $p\ge 5$ then please skip this subsection. 
\begin{lem}\label{p3} Assume $p=3$ and let $N$ be as in Proposition \ref{extX} then the composition:
\begin{displaymath}
\xymatrix@1{ \Ext^1_{I/Z_1}(\chi\alpha, N/\chi)\ar[r]^-\partial & \Ext^2_{I/Z_1}(\chi\alpha,\chi)\ar[r]^-{\Res} & \Ext^2_{(I\cap P)/Z_1}(\chi\alpha,\chi)}
\end{displaymath}
is injective, where $\partial$ is induced by a short exact sequence $0\rightarrow \chi\rightarrow N\rightarrow N/\chi\rightarrow 0$.
\end{lem}
\begin{proof} Since $p=3$ we have $\alpha=\alpha^{-1}$ and hence it follows from the Corollary \ref{take} that 
$\dim \Ext^1_{I/Z_1}(\chi\alpha, N/\chi)=1$. Corollary  \ref{extX} implies that  the restriction map 
$\Ext^1_{I/Z_1}(\chi\alpha, N/\chi)\rightarrow \Ext^1_{(I\cap P)/Z_1}(\chi\alpha, N/\chi)$ is injective. Moreover, Lemma \ref{extXP} gives
$\Ext^1_{(I\cap P)/Z_1}(\chi\alpha, N)=0$, and so the map 
$\partial:  \Ext^1_{(I\cap P)/Z_1}(\chi\alpha, N/\chi) \rightarrow   \Ext^2_{(I\cap P)/Z_1}(\chi\alpha,\chi)$ is injective. 
The assertion follows from the commutative diagram:
\begin{displaymath}
\xymatrix@1{ \Ext^1_{I/Z_1}(\chi\alpha, N/\chi)\ar[r]^-\partial\ar@{^(->}[d]^{\Res}_{\eqref{natural2}} & \Ext^2_{I/Z_1}(\chi\alpha,\chi)\ar[d]^{\Res}\\
        \;\Ext^1_{(I\cap P)/Z_1}(\chi\alpha, N/\chi)\; \ar@{^{(}->}[r]^-{\partial} &  \Ext^2_{(I\cap P)/Z_1}(\chi\alpha,\chi).}
\end{displaymath}
\end{proof}

\begin{lem}\label{p3bis} Assume $p=3$ and let $N$ be as in Proposition \ref{extX}. Assume that 
$N^{K_1}\cong \Indu{HK_1}{I}{\chi}$ as a representation of $I$, 
 then the composition:
\begin{displaymath}
\xymatrix@1{ \Ext^1_{I/Z_1}(\chi\alpha, N/\chi)\ar[r]^-\partial & \Ext^2_{I/Z_1}(\chi\alpha,\chi)\ar[r]^-{\Res} 
& \Ext^2_{(I\cap P^s)/Z_1}(\chi\alpha,\chi)}
\end{displaymath}
is zero, where $\partial$ is induced by a short exact sequence $0\rightarrow \chi\rightarrow N\rightarrow N/\chi\rightarrow 0$.
\end{lem}
\begin{proof}  Since $p=3$ we have $\alpha=\alpha^{-1}$ and hence it follows from the Corollary \ref{take} that 
$\dim \Ext^1_{I/Z_1}(\chi\alpha, N/\chi)=1$. Let $\Delta$ be the image of the restriction map 
$$\Delta:=\Image(\Ext^1_{I/Z_1}(\chi\alpha, N/\chi)\rightarrow \Ext^1_{(I\cap P^s)/Z_1}(\chi\alpha, N/\chi)).$$ 
We claim that $\Delta$ is 
contained in the image of the natural map 
\begin{equation}\label{map}
\Ext^1_{(I\cap P^s)/Z_1}(\chi\alpha, N)\rightarrow \Ext^1_{(I\cap P^s)/Z_1}(\chi\alpha, N/\chi).
\end{equation}
Since $p=3$ we have $\dim N^{K_1}=3$ and so the image of $N^{K_1}$ in $N/N^{I_1}$ is a $2$-dimensional $I$-stable subspace. 
Since it follows from Proposition \ref{modX} that $(N/N^{I_1})^{I_1}$ and $((N/N^{I_1})/(N/N^{I_1})^{I_1})^{I_1}$ are $1$-dimensional 
 we obtain an exact sequence $0\rightarrow N^{I_1}\rightarrow N^{K_1}\rightarrow R\rightarrow 0$,
where where $R$ is the subspace of $N/\chi$ defined in  Proposition \ref{extX} (with $N/\chi$ instead of $N$). 
Since  $N^{K_1}\cong \Indu{HK_1}{I}{\chi}$
we get:
$$N^{K_1}|_{I\cap P^s}\cong \chi \oplus \chi\alpha \oplus \chi \cong \chi \oplus R|_{I\cap P^s}.$$
Let $\phi$ be the composition: 
\begin{equation}\notag
\begin{split}
 \Ext^1_{(I\cap P^s)/Z_1}&(\chi\alpha, R)\rightarrow \Ext^1_{(I\cap P^s)/Z_1}(\chi\alpha, N^{K_1})\rightarrow 
\\&\Ext^1_{(I\cap P^s)/Z_1}(\chi\alpha, N)
\rightarrow \Ext^1_{(I\cap P^s)/Z_1}(\chi\alpha, N/\chi).
\end{split}
\end{equation}
Then we have a commutative diagram:
\begin{displaymath}
\xymatrix{  \Ext^1_{I/Z_1}(\chi\alpha, R)\ar@{>>}[r]^-{\ref{extX}}\ar[d]^{\Res} & \Ext^1_{I/Z_1}(\chi\alpha, N/\chi)\ar[d]^{\Res}\\
\Ext^1_{(I\cap P^s)/Z_1}(\chi\alpha, R)\ar[r]^-{\phi} &  \Ext^1_{(I\cap P^s)/Z_1}(\chi\alpha, N/\chi).}
\end{displaymath}
The top horizontal arrow is surjective by Proposition \ref{extX}. Hence, $\Delta$ equals to the image of $\phi\circ \Res$. Since 
the image of $\phi$ is contained in the image of \eqref{map} we get the claim. The assertion follows from the commutative diagram:
\begin{displaymath}
\xymatrix@1{ \Ext^1_{I/Z_1}(\chi\alpha, N/\chi)\ar[r]^-\partial\ar[d]^{\Res} & \Ext^2_{I/Z_1}(\chi\alpha,\chi)\ar[d]^{\Res}\\
         \Ext^1_{(I\cap P^s)/Z_1}(\chi\alpha, N/\chi) \ar[r]^-{\partial} &  \Ext^2_{(I\cap P^s)/Z_1}(\chi\alpha,\chi),}
\end{displaymath}   
since the claim implies that the composition $\partial \circ \Res$ is the zero map. 
\end{proof}

\begin{lem}\label{done} Assume $p=3$ let $N_{\chi}$ and $N_{\chi^s}$ be as in Proposition \ref{extX} with respect to  $\chi$ and $\chi^s$.
Further assume that $N_{\chi^s}^{K_1}\cong \Indu{HK_1}{I}{\chi^s}$ as a representation of $I$, then the natural map 
\begin{equation}\label{injExt2}
\Ext^2_{I/Z_1}(\chi\alpha,\chi)\rightarrow \Ext^2_{I/Z_1}(\chi\alpha, N_{\chi})\oplus \Ext^2_{I/Z_1}(\chi\alpha, N_{\chi^s}^{\Pi})
\end{equation}
is injective, where $N_{\chi^s}^\Pi$ denotes the twist of action of $I$ on $N_{\chi^s}$ by $\Pi$.
\end{lem}
\begin{proof} Applying $\Hom_{I/Z_1}(\chi\alpha, \centerdot)$ to the short exact sequence $0\rightarrow \chi\rightarrow N_{\chi}\rightarrow N_{\chi}/\chi\rightarrow 0$
gives a long exact sequence. Equation \eqref{yetanotherseq} shows that the map $\Ext^1_{I/Z_1}(\chi\alpha,\chi)\rightarrow \Ext^1_{I/Z_1}(\chi\alpha,N_\chi)$
is surjective, which implies that  
$$\Ker(\Ext^2_{I/Z_1}(\chi\alpha,\chi)\rightarrow \Ext^2_{I/Z_1}(\chi\alpha, N_{\chi}))\cong \Ext^1_{I/Z_1}(\chi\alpha, N_{\chi}/\chi).$$
If we replace $N_{\chi}$ with $N_{\chi^s}$ and $\chi$ with $\chi^s$ the same isomorphism holds. Twisting by $\Pi$ gives:
$$\Ker(\Ext^2_{I/Z_1}(\chi\alpha,\chi)\rightarrow \Ext^2_{I/Z_1}(\chi\alpha, N_{\chi^s}^{\Pi}))\cong \Ext^1_{I/Z_1}(\chi\alpha, N_{\chi^s}^{\Pi}/\chi).$$
Lemma \ref{p3} implies that the  composition 
$$\Res\circ \partial: \Ext^1_{I/Z_1}(\chi\alpha, N_{\chi}/\chi)\rightarrow  \Ext^2_{(I\cap P)/Z_1}(\chi\alpha,\chi)$$
is an injection. And Lemma \ref{p3bis} implies that the composition
$$\Res\circ\partial: \Ext^1_{I/Z_1}(\chi\alpha, N_{\chi^s}^{\Pi}/\chi)\rightarrow \Ext^2_{(I\cap P)/Z_1}(\chi\alpha,\chi)$$
is zero. Hence, $\partial( \Ext^1_{I/Z_1}(\chi\alpha, N_{\chi}/\chi))\cap \partial(\Ext^1_{I/Z_1}(\chi\alpha, N_{\chi}^{\Pi}/\chi))=0$ 
and so the map in \eqref{injExt2} is injective.
\end{proof}

\begin{lem}\label{p3r0} Assume $p=3$ and $r=0$ then $M_{\tsigma}$ satisfies the assumptions of Lemma \ref{p3bis}. 
\end{lem}
\begin{proof} Now $\langle (I\cap U)tv_{\sigma}\rangle =\langle I sv_{\tsigma}\rangle \cong \St|_I\cong \Indu{HK_1}{I}{\chi^s}$ as a representation of $I$,
where $\St\cong \Sym^2\overline{\mathbb F}_3^2$ is the Steinberg representation of $\GL_2(\mathbb{F}_3)$.
Hence we have an injection $\Indu{HK_1}{I}{\chi^s}\hookrightarrow M_{\tsigma}$. Since $M_{\tsigma}|_{H(I\cap U)}$ is an injective envelope 
of $\chi^s$ in $\Rep_{H(I\cap U)}$ we obtain that $M_{\tsigma}^{K_1\cap U}\cong \Indu{H(K_1\cap U)}{H(I\cap U)}{\chi^s}$ as
a representation of $H(I\cap U)$. Hence $\dim M_{\tsigma}^{K_1\cap U}=3$ and so we obtain $M_{\tsigma}^{K_1\cap U}\cong M_{\tsigma}^{K_1}\cong 
\Indu{HK_1}{I}{\chi^s}$.
\end{proof}
\subsection{}
Using the Lemmas above we prove the main result of this section.
\begin{thm}\label{main} Assume $p>2$ and let $\psi: H\rightarrow \Fbar^{\times}$ be a character, such that $\Ext^1_{I/Z_1}(\psi,\pi_{\sigma})\neq 0$. 
Then $\psi\in \{\chi\alpha, \chi\}$. Moreover, 
\begin{itemize}
\item[(i)] $\dim \Ext^1_{I/Z_1}(\chi,\pi_{\sigma})=2$;
\item[(ii)] if $p>3$ or $p=3$ and $r\in \{0,2\}$ then  $\Ext^1_{I/Z_1}(\chi\alpha,\pi_{\sigma})=0$;
\item[(iii)] if $p=3$ and $r=1$ then $\dim  \Ext^1_{I/Z_1}(\chi\alpha,\pi_{\sigma})\le 1$.
\end{itemize}
\end{thm}
\begin{proof} Corollary  \ref{take}, \eqref{natural1} gives injections: 
$$\Ext^2_{I/Z_1}(\chi,\chi)\hookrightarrow \Ext^2_{I/Z_1}(\chi,M_{\sigma}),$$  
$$\Ext^2_{I/Z_1}(\chi,\chi)\hookrightarrow \Ext^2_{I/Z_1}(\chi,\Pi\centerdot M_{\tsigma}).$$ 
 Moreover, $\Ext^1_{I/Z_1}(\chi,\chi)=0$.
Corollary \ref{one} gives a short exact sequence 
$0\rightarrow \chi\rightarrow M_{\sigma}\oplus \Pi\centerdot M_{\tsigma}\rightarrow \pi_{\sigma}\rightarrow 0$,
which induces an isomorphism:
$$\Ext^1_{I/Z_1}(\chi,M_{\sigma})\oplus \Ext^1_{I/Z_1}(\chi,\Pi\centerdot M_{\tsigma})\cong \Ext^1_{I/Z_1}(\chi,\pi_{\sigma}).$$ 
Corollary \ref{take} implies that $\dim \Ext^1_{I/Z_1}(\chi,\pi_{\sigma})=2$. 

Assume that $\psi\neq \chi$. From  
$0\rightarrow M_{\sigma}\rightarrow \pi_{\sigma}\rightarrow (\Pi\centerdot M_{\tsigma})/\chi\rightarrow 0$ we obtain a long exact sequence:
\begin{displaymath}
\begin{split}
\Hom_{I}(\psi,\chi\alpha)\hookrightarrow \Ext^1_{I/Z_1}(\psi,M_{\sigma})\rightarrow &\Ext^1_{I/Z_1}(\psi,\pi_{\sigma})\rightarrow\\ 
&\Ext^1_{I/Z_1}(\psi,(\Pi\centerdot M_{\tsigma})/\chi).
\end{split}
\end{displaymath}
If $\Ext^1_{I/Z_1}(\psi,M_{\sigma})\neq 0$ then Proposition \ref{extX} implies 
$\psi=\chi\alpha$. Similarly, if $\Ext^1_{I/Z_1}(\psi,(\Pi\centerdot M_{\tsigma})/\chi)\neq 0$
then $\psi=(\chi^s \alpha^{-1})^{\Pi}=\chi\alpha$. Hence, $\psi=\chi\alpha$ and $\dim \Ext^1_{I/Z_1}(\chi\alpha, \pi_\sigma)\le 1$.

If $p>3$ then Proposition \ref{extX} implies that $\Ext^1_{I/Z_1}(\chi\alpha, M_{\sigma}/\chi)=0$. Hence
the exact sequence $0\rightarrow \Pi\centerdot M_{\tsigma}\rightarrow \pi_{\sigma}\rightarrow M_{\sigma}/\chi\rightarrow 0$
gives an exact sequence:
$$\Hom_{I}(\chi\alpha,\chi\alpha^{-1})\hookrightarrow \Ext^1_{I/Z_1}(\chi\alpha,\Pi\centerdot M_{\tsigma})\twoheadrightarrow \Ext^1_{I/Z_1}(\chi\alpha,\pi_{\sigma}).$$
Since $p>3$ Proposition \ref{extX} implies that $\Ext^1_{I/Z_1}(\chi\alpha,\Pi\centerdot M_{\tsigma})=0$ and hence $\Ext^1_{I/Z_1}(\chi\alpha,\pi_{\sigma})=0$.

Assume that $p=3$ and  $r=0$ Lemmas \ref{done} and \ref{p3r0} give an exact sequence:
$$ \Ext^1_{I/Z_1}(\chi\alpha, \chi)\hookrightarrow \Ext^1_{I/Z_1}(\chi\alpha, M_{\sigma}\oplus \Pi\centerdot M_{\tsigma})\twoheadrightarrow  
\Ext^1_{I/Z_1}(\chi\alpha, \pi_\sigma).$$
Since $p=3$ we have  $\dim  \Ext^1_{I/Z_1}(\chi\alpha, \chi)=2$ 
and Proposition \ref{extX} gives $\dim \Ext^1_{I/Z_1}(\chi\alpha, M_{\sigma}\oplus \Pi\centerdot M_{\tsigma})=2$.
Hence $\Ext^1_{I/Z_1}(\chi\alpha, \pi_\sigma)=0$. Since $p=3$ and $r=0$ we have $(\chi\alpha)^{\Pi}=\chi\alpha$, $\chi=\chi^s$ and since 
$\pi_{\tsigma}=\Pi\centerdot \pi_{\sigma}$, we also obtain $\Ext^1_{I/Z_1}(\chi\alpha, \pi_{\tsigma})=0$, which deals with the case $p=3$ and $r=2$.
\end{proof} 

\begin{cor}\label{dimH1} Assume $p>2$ and let $\psi:H\rightarrow \Fbar^{\times}$ be a character. Suppose that $\Hom_I(\psi, H^1(I_1/Z_1, \pi))\neq 0$ then 
$\psi\in \{\chi, \chi^s\}$. Moreover, the following hold:
\begin{itemize} 
\item[(i)] if $p=3$ and $r=1$ then $\dim H^1(I_1/Z_1, \pi)\le 6$;
\item[(ii)] otherwise, $\dim H^1(I_1/Z_1,\pi)=4$.
\end{itemize}
\end{cor} 
\begin{proof} By Corollary \ref{two} $\pi\cong \pi_{\sigma}\oplus \pi_{\tsigma}$ as $I$-representations. 
The assertion follows from Theorem \ref{main}. We note that if $p=3$ and 
$r=1$ then $\chi\alpha=\chi^s$ and $\chi^s\alpha=\chi$. 
\end{proof}

\section{Extensions and central characters}\label{extandcenter}
We fix a smooth character $\zeta: Z\rightarrow \Fbar^{\times}$ and let $\Rep_{G, \zeta}$ be the full category of $\Rep_G$ consisting 
of representations with central character $\zeta$. Let $V$ be an $\Fbar$-vector space with an action of $Z$, given by 
$zv=\zeta(z)v$, for all $z\in Z$ and $v\in V$. Then $\Indu{Z}{G}{V}$ is an object of $\Rep_{G, \zeta}$, moreover given $\pi$ in $\Rep_{G, \zeta}$ by Frobenius 
reciprocity we get  
\begin{equation}\label{Frobrec}
\Hom_G(\pi, \Indu{Z}{G}{V})\cong \Hom_Z(\pi, V)\cong \Hom_{\Fbar}(\pi, V).
\end{equation}
Hence, the functor $\Hom_G(\centerdot, \Indu{Z}{G}{V})$ is exact and so $\Indu{Z}{G}{V}$ is an injective object in $\Rep_{G, \zeta}$. Further, 
if $V$ is the underlying vector space of $\pi$ then we may embed $\pi\hookrightarrow \Indu{Z}{G}{V}$, $v\mapsto [g\mapsto g v]$. Hence, 
$\Rep_{G, \zeta}$ has enough injectives.
  
For  $\pi_1, \pi_2$ in $\Rep_{G, \zeta}$ we denote $\Ext^1_{G,\zeta}(\pi_1,\pi_2):=\RR^1\Hom(\pi_1,\pi_2)$ 
computed in the category $\Rep_{G,\zeta}$.  
\begin{prop}\label{centre} Let $\pi_1$ and $\pi_2$ be irreducible representations of $G$ admitting a central character. 
Let $\zeta$ be the central character of $\pi_2$. If $\Ext^1_G(\pi_1,\pi_2)\neq 0$ then $\zeta$ is also the central character of $\pi_1$.
If $\pi_1\not\cong \pi_2$ then $\Ext^1_{G,\zeta}(\pi_1,\pi_2)=\Ext^1_{G}(\pi_1,\pi_2)$. If $\pi_1\cong\pi_2$ then 
there exists an exact sequence:
$$0\rightarrow \Ext^1_{G,\zeta}(\pi_1,\pi_2)\rightarrow \Ext^1_{G}(\pi_1,\pi_2)\rightarrow \Hom(Z,\Fbar)\rightarrow 0.$$
\end{prop}
\begin{proof}Suppose that we have a non-split extension $0\rightarrow \pi_2\rightarrow E\rightarrow \pi_1\rightarrow 0$ in $\Rep_G$.
For all $z\in Z$ we define $\theta_z: E\rightarrow E$, $v\mapsto z v -\zeta(z)v$. Since $z$ is central in $G$, $\theta_z$ is $G$-equivariant. 
If $\theta_z=0$ for all $z\in Z$ then $E$ admits a central character $\zeta$, and hence $\zeta$ is the central character of $\pi_1$ and 
the extension lies in $\Ext^1_{G,\zeta}(\pi_1,\pi_2)$. 
If $\theta_z\neq 0$ for some $z\in Z$ then it induces an isomorphism $\pi_1\cong \pi_2$.

We assume that $\pi_1\cong\pi_2$ and drop the subscript. Then \eqref{Frobrec} gives $\Hom_G(\pi,\Indu{Z}{G}{\zeta})\cong\pi^*$.
Fix a non-zero $\varphi\in \Hom_Z(\pi,\zeta)$. Since $\pi$ is irreducible we obtain an exact sequence:
\begin{equation}\label{seqcen}
0\rightarrow \pi\overset{\varphi}{\rightarrow} \Indu{Z}{G}{\zeta}\rightarrow Q\rightarrow 0.
\end{equation}
Since $\Indu{Z}{G}{\zeta}$ is an injective object in $\Rep_{G,\zeta}$, and \eqref{seqcen} is in $\Rep_{G,\zeta}$ by applying  $\Hom_G(\pi, \centerdot )$ to 
\eqref{seqcen} we obtain an exact sequence:
\begin{equation}\label{seq1}
\pi^*\rightarrow \Hom_G(\pi, Q)\rightarrow \Ext^1_{G,\zeta}(\pi,\pi)\rightarrow 0.
\end{equation}
If we consider \eqref{seqcen} as an exact sequence in $\Rep_G$ then by applying $\Hom_G(\pi, \centerdot )$ we get  an exact sequence:
\begin{equation}\label{seq2}
\pi^*\rightarrow \Hom_G(\pi, Q)\rightarrow \Ext^1_{G}(\pi,\pi)\rightarrow \Ext^1_{G}(\pi,  \Indu{Z}{G}{\zeta}).
\end{equation}
Putting \eqref{seq1} and \eqref{seq2} together we obtain an exact sequence:
$$0\rightarrow \Ext^1_{G,\zeta}(\pi,\pi)\rightarrow \Ext^1_G(\pi,\pi)\rightarrow \Ext^1_{G}(\pi,  \Indu{Z}{G}{\zeta}).$$

Let $0\rightarrow  \Indu{Z}{G}{\zeta}\rightarrow E\rightarrow \pi\rightarrow 0$ be an extension in $\Rep_G$. For all $z\in Z$, $\theta_z:E\rightarrow E$ 
induces $\theta_z(E)\in \Hom_G(\pi,  \Indu{Z}{G}{\zeta})$. Now $\theta_z(E)=0$ for all $z\in Z$ if and only if $E$ has 
a central character $\zeta$, but since $\Indu{Z}{G}{\zeta}$ is an injective object in $\Rep_{G,\zeta}$ Lemma \ref{extriv} implies that the sequence is split 
if and only if $E$ has a central character $\zeta$. Now 
\begin{equation}\label{thetapsi}
\begin{split}
\theta_{z_1z_2}(v)&= z_1z_2 v-\zeta(z_1z_2)v= z_1( z_2 v-\zeta(z_2)v)+ z_1\zeta(z_2)v-\zeta(z_1z_2)v\\&= \zeta(z_1)\theta_{z_2}(v)+\zeta(z_2)\theta_{z_1}(v).
\end{split}
\end{equation}
Hence, if we set  $\psi_E(z):=\zeta(z)^{-1} \theta_z(E)$, then \eqref{thetapsi} gives $\psi_E(z_1z_2)=\psi_E(z_1)+\psi_E(z_2)$. 
 Hence, the map $E\mapsto \psi_E$ induces an injection 
 $\Ext^1_{G}(\pi,  \Indu{Z}{G}{\zeta})\hookrightarrow \Hom(Z, \pi^*)$. The image 
of $$\Ext^1_G(\pi,\pi)\rightarrow \Ext^1_{G}(\pi,  \Indu{Z}{G}{\zeta})\hookrightarrow \Hom(Z, \pi^*)$$
is $\Hom(Z, \Fbar\varphi)$, which is isomorphic to $\Hom(Z,\Fbar)$.
\end{proof}

\begin{prop}\label{CK} Let $\pi:=\pi(r, 0, \eta)$ and $\zeta$ the central character of $\pi$. Assume that $p>2$ and $(p,r)\neq (3,1)$ then 
$\dim \Ext^1_{G, \zeta}(\pi,\pi)\ge 3$.
\end{prop}
\begin{proof} This follows from \cite[2.3.4]{kis1}.
\end{proof}
\begin{remar} At the time of writing this note, \cite{kis1} was not written up and there were some technical issues 
with the outline of the argument given in the introductions to \cite{col1} and \cite{kis}. Since we  only need 
a lower bound on the dimension and only in the supersingular case, we have written up  another proof of Proposition 
\ref{CK} in the appendix.  The proof given there is a variation of Colmez-Kisin argument.    
\end{remar}

\section{Hecke Algebra}\label{hecke}

Let $\zeta$ be the central character of $\pi$. Let $\HH:=\End_{G}(\cIndu{ZI_1}{G}{\zeta})$.
 Let $\II: \Rep_{G,\zeta}\rightarrow \Mod_{\HH}$ be the functor: 
$$\II(\pi):=\pi^{I_1}\cong \Hom_G(\cIndu{ZI_1}{G}{\zeta}, \pi).$$
Let $\TT:\Mod_{\HH}\rightarrow \Rep_{G,\zeta}$ be the functor:
$$\TT(M):=M\otimes_{\HH} \cIndu{ZI_1}{G}{\zeta}.$$
One has $\Hom_{\HH}(M, \II(\pi))\cong \Hom_G(\TT(M),\pi)$. Moreover, Vign\'{e}ras 
in \cite[Thm.5.4]{vig} shows that $\II$ induces a bijection between irreducible objects in $\Rep_{G,\zeta}$ and $\Mod_{\HH}$. 
Let $\Rep^{I_1}_{G, \zeta}$ be the  full subcategory of  $\Rep_{G,\zeta}$ consisting 
of representations generated by their $I_1$-invariants.  Ollivier 
has shown \cite{o2} that 
\begin{equation}
\II: \Rep^{I_1}_{G, \zeta}\rightarrow \Mod_{\HH}, \quad \TT: \Mod_{\HH}\rightarrow \Rep^{I_1}_{G, \zeta}
\end{equation} 
are quasi-inverse to each other and so $\Mod_{\HH}$ is equivalent to   $\Rep^{I_1}_{G, \zeta}$. In particular, 
suppose that $\tau=\langle G \centerdot \tau^{I_1}\rangle$,  $\pi$ 
in $\Rep_{G, \zeta}$  and let $\pi_1:=\langle G\centerdot \pi^{I_1}\rangle \subseteq \pi$ then one has:
\begin{equation}\label{homit}
\begin{split}
\Hom_G(\tau, \pi)\cong  \Hom_{G}(\tau, \pi_1) \cong  \Hom_{\HH}&(\II(\tau), \II(\pi_1))\\ &\cong  \Hom_{\HH}(\II(\tau), \II(\pi))
\end{split}
\end{equation}
and the natural map $\TT \II(\tau)\rightarrow \tau$  is an isomorphism.

Let $J$ be an injective object in $\Rep_{G, \zeta}$, then the first isomorphism of \eqref{homit} 
implies that $J_1:=\langle G\centerdot J^{I_1}\rangle$ is an injective object in $\Rep^{I_1}_{G, \zeta}$. Since $\TT$ and $\II$ 
induce an equivalence of categories between $\Mod_{\HH}$ and $\Rep^{I_1}_{G, \zeta}$ we obtain that $\II(J_1)=\II(J)$ is an injective 
object in $\Mod_{\HH}$. Hence, \eqref{homit} gives an $E_2$-spectral sequence:
\begin{equation}\label{specseq}
\Ext^i_{\HH}(\II(\tau), \RR^j \II(\pi))\Longrightarrow \Ext^{i+j}_{G, \zeta}(\tau, \pi)
\end{equation}
The $5$-term sequence associated to \eqref{specseq} gives us:
\begin{prop}\label{5term} Let $\tau$ and $\pi$ be in $\Rep_{G, \zeta}$ suppose that $\tau$ is generated as a $G$-representation by $\tau^{I_1}$
then there exists an exact sequence:
\begin{equation}\label{5T}
\begin{split}
0\rightarrow &\Ext^1_{\HH}(\II(\tau), \II(\pi))\rightarrow \Ext^1_{G, \zeta}(\tau, \pi)\rightarrow \Hom_{\HH}(\II(\tau), \RR^1\II(\pi))\\
&\rightarrow \Ext^2_{\HH}(\II(\tau), \II(\pi))\rightarrow \Ext^2_{G, \zeta}(\tau, \pi)
\end{split}
\end{equation}
\end{prop}

It is easy to write down the first two non-trivial arrows of \eqref{5T} explicitly. An extension class 
of $0\rightarrow \II(\pi)\rightarrow E\rightarrow \II(\tau)\rightarrow 0$ maps to the extension 
class of $0\rightarrow \TT\II(\pi)\rightarrow \TT(E)\rightarrow \TT\II(\tau)\rightarrow 0$. Let    
  $\epsilon$ be an extension class of $0\rightarrow \pi\rightarrow \kappa\rightarrow \tau\rightarrow 0$. We may apply $\II$ to get  
\begin{equation}\label{nearlydone}
\xymatrix@1{0 \ar[r]& \II(\pi)\ar[r]& \II(\kappa)\ar[r]& \II(\tau)\ar[r]^-{\partial_{\epsilon}}& \mathbb R^1\II(\pi).}
\end{equation}
The second non-trivial arrow in \eqref{5T} is given by $\epsilon\mapsto \partial_{\epsilon}$.

We are interested in \eqref{specseq} when both $\pi$ and $\tau$ are irreducible.
 We recall some 
facts about the structure of $\HH$ and its irreducible modules, for proofs see \cite{vig} or \cite[\S1]{coeff}. As an $\Fbar$-vector space 
$\HH$ has a basis indexed by double cosets $I_1\backslash G/Z I_1$, we write $T_g$ for the element corresponding to a double coset $I_1g Z I_1$.
Given $\pi$ in $\Rep_{G,\zeta}$, and $v\in \pi^{I_1}$, the action of $T_g$ is given  by:
\begin{equation}\label{Tg}
vT_g= \sum_{u\in I_1/(I_1\cap g^{-1} I_1 g)} u g^{-1} v.
\end{equation}
Let $\chi:H\rightarrow \Fbar^{\times}$ be a character then we define $e_{\chi}\in \HH$ by 
$$e_{\chi}:=\frac{1}{|H|} \sum_{h\in H} \chi(h) T_h.$$ 
Then $e_{\chi}e_{\psi}=e_{\chi}$ if $\chi=\psi$ and $0$ otherwise and it follows from \eqref{Tg} that $\pi^{I_1}e_{\chi}$ is the $\chi$-isotypical 
subspace of 
$\pi^{I_1}$ as a representation of $H$. The elements $T_{n_s}$, $T_{\Pi}$ and $e_{\chi}$, for all $\chi$ generate $\HH$ as an algebra, 
and are subject to the following relations:  $T_{\Pi}^2=\zeta(p)^{-1}$,
\begin{equation}\label{relTH}
e_{\chi}T_{n_s}=T_{n_s}e_{\chi^s}, \quad e_{\chi}T_{\Pi}=T_{\Pi}e_{\chi^s}, \quad   \ e_{\chi}T_{n_s}^2=-e_{\chi}e_{\chi^s}T_{n_s}.
\end{equation}
Note that $e_{\chi}e_{\chi^s}=e_{\chi}$ if $\chi=\chi^s$ and $e_{\chi}e_{\chi^s}=0$, otherwise.  We let $\HH^+$ be the subalgebra of $\HH$ generated by 
$T_{n_s}$, $T_{\Pi} T_{n_s} T_{\Pi}^{-1}$ and $e_{\chi}$ for all characters $\chi$. One may naturally identify $\HH^+\cong \End_{G^+}(\cIndu{ZI_1}{G^+}{\zeta})$.

\begin{defi}\label{Mrlambdaeta} Let $0\le r \le p-1$ be an integer, $\lambda\in \Fbar$ and $\eta: \Qp^{\times}\rightarrow \Fbar^{\times}$ a smooth character,
and let $\zeta$ be the central character of $\pi(r,\lambda,\eta)$ then we 
define $\HH$-modules $M(r,\lambda):= \pi(r,\lambda)^{I_1}$, $M(r,\lambda, \eta):=\pi(r,\lambda, \eta)^{I_1}$.
\end{defi}
Assume for simplicity that $\zeta(p)=1$ then it is shown in \cite[Cor. 6.4]{bp} that $M(r,\lambda, \eta)$ has an $\Fbar$-basis 
$\{v_1, v_2\}$ such that  
\begin{itemize}
\item[(i)]$v_1 e_{\chi}= v_1,\quad  v_1 T_{\Pi}= v_2, \quad v_2 e_{\chi^s}= v_2, \quad v_2 T_{\Pi}=v_1$
and such that $v_1 T_{n_s}=-v_1$ if $r=p-1$ and $v_1 T_{n_s}=0$ otherwise.
\item[(ii)]$v_2(1+T_{n_s})= \eta(-p^{-1})\lambda v_1$ if $r=0$ and $v_2 T_{n_s}= \eta(-p^{-1})\lambda v_1$ otherwise,
\end{itemize}
where $\chi:H\rightarrow \Fbar^{\times}$ is the character $\chi(\bigl (\begin{smallmatrix} [\lambda]& 0\\ 0 & [\mu]\end{smallmatrix}\bigr))= 
\lambda^r\eta([\lambda\mu])$. If $\lambda=0$ so that $\pi(r,\lambda,\eta)$ is supersingular, then $v_1=v_{\sigma}$ and $v_2=v_{\tsigma}$.

\begin{lem}\label{extH} Let $\pi$ be a supersingular representation of $G$ then
\begin{itemize}
\item[(i)] if $r\in \{0,p-1\}$ then 
\begin{itemize}
\item[(a)] $\dim \Ext^1_{\HH}(\pi^{I_1},\pi^{I_1})=1$;
\item[(b)]  $\Ext^i_{\HH}(\pi^{I_1}, \ast)=0$ for $i>1$;
\end{itemize}
\item[(ii)] otherwise, $\dim \Ext^1_{\HH}(\pi^{I_1},\pi^{I_1})=2$.
\end{itemize}
\end{lem}
\begin{proof}  \cite[Cor. 6.7, 6.6]{bp}. 
\end{proof}

We look more closely at the regular case. Let $\pi$ be supersingular with $0<r<p-1$ and assume for 
simplicity that $p\in Z$ acts trivially on $\pi$. For $(\lambda_1, \lambda_2)\in \Fbar^2$ we define 
an $\HH$-module  $E_{\lambda_1, \lambda_2}$ to be a $4$-dimensional vector 
space with basis $\{v_{\chi}, v_{\chi^s}, w_{\chi}, w_{\chi^s}\}$ with the action of $\HH$ given on the generators
\begin{equation}
w_{\chi} T_{n_s}= \lambda_1 v_{\chi^s}, \quad w_{\chi^s}T_{n_s}= \lambda_2 v_{\chi}, \quad 
v_{\chi} T_{n_s}= v_{\chi^s} T_{n_s}=0
\end{equation}
and $w_\psi T_{\Pi}=w_{\psi^s}$, $v_{\psi}T_{\Pi}= v_{\psi^s}$, $w_{\psi} e_{\psi}= w_{\psi}$, 
$v_{\psi} e_{\psi}= v_{\psi}$, for $\psi\in \{\chi, \chi^s\}$. Then 
$\langle v_{\chi}, v_{\chi^s}\rangle$ is stable under the action of $\HH$ and we have an exact sequence:
\begin{equation}\label{stex}
0\rightarrow \II(\pi)\rightarrow E_{\lambda_1, \lambda_2}\rightarrow \II(\pi)\rightarrow 0 
\end{equation}
The extension \eqref{stex} is split if and only if $(\lambda_1,\lambda_2)=(0,0)$. It is immediate that 
the map $\Fbar^2\rightarrow  \Ext^1_{\HH}(\II(\pi), \II(\pi))$ 
sending $(\lambda_1, \lambda_2)$ to the equivalence class of \eqref{stex} 
is an isomorphism of $\Fbar$-vector spaces. 

\begin{lem}\label{lastdino} Let $\lambda\in \Fbar^{\times}$ then 
\begin{equation}\label{lastdino1}
\TT(E_{0,\lambda})\cong \frac{\cIndu{KZ}{G}{\sigma}}{(T_{\sigma}^2)}, \quad  \TT(E_{\lambda,0})\cong \frac{\cIndu{KZ}{G}{\tsigma}}{(T_{\tsigma}^2)}
\end{equation}
where $T_{\sigma}\in \End_G(\cIndu{KZ}{G}{\sigma})$ is given  by Lemma \ref{T}.
\end{lem}
\begin{proof} Let $\varphi\in \cIndu{KZ}{G}{\sigma}$ such that $\supp \varphi=KZ$ and $\varphi(1)$ spans $\sigma^{I_1}$. 
Let $\tau:= \frac{\cIndu{KZ}{G}{\sigma}}{(T_{\sigma}^2)}$ and $v$ the image of $\varphi$ in $\tau$. Then 
$\tau=\langle G\centerdot v\rangle=\langle G\centerdot \tau^{I_1}\rangle$. And so it is enough to show that 
$\II(\tau)\cong E_{0, \lambda}$. Since $T_{\sigma}: \cIndu{KZ}{G}{\sigma}\rightarrow \cIndu{KZ}{G}{\sigma}$ is 
injective and $\pi\cong \frac{\cIndu{KZ}{G}{\sigma}}{(T_{\sigma})}$, we have a an exact sequence 
\begin{equation}\label{lastdino2}
0\rightarrow \pi\rightarrow \tau\rightarrow \pi\rightarrow 0
\end{equation}
 and we may identify the subobject with $T_{\sigma}(\tau)$. Now, $v$, $\Pi v$, $T_{\sigma}(v)$ and $T_{\sigma}(\Pi v)$
are linearly independent and $I_1$-invariant. Thus $\dim \tau^{I_1}\ge 4$ and since $\dim \pi^{I_1}=2$ we obtain an 
exact sequence of $\HH$-modules 
\begin{equation}\label{lastdino3}
0\rightarrow \II(\pi)\rightarrow \II(\tau)\rightarrow \II(\pi)\rightarrow 0
\end{equation}
Hence, $\II(\tau)\cong E_{\lambda_1, \lambda_2}$ for some $\lambda_1, \lambda_2\in \Fbar$. Since 
$\sigma\cong \langle K\centerdot \varphi\rangle\cong  \langle K\centerdot v\rangle$  and 
$\langle K\centerdot T_{\sigma}(v)\rangle\cong  T_{\sigma}(\langle K\centerdot v\rangle)\cong \sigma$, \cite[3.1.3]{coeff}
gives
\begin{equation}\label{lastdino4}
v e_{\chi}= v, \quad (T_{\sigma}(v))e_{\chi}=T_{\sigma}(v), \quad v T_{n_s}= (T_{\sigma}(v))T_{n_s}=0.
\end{equation}
Hence, $\lambda_1=0$. If $\lambda_2=0$ then \eqref{lastdino3} would split and so would  \eqref{lastdino2}. Hence, $\lambda_2\neq 0$. 
We leave it to the reader to check that for any $\lambda\in \Fbar^{\times}$, $E_{0,\lambda}\cong E_{0,1}$. 
\end{proof}

\begin{lem}\label{kuku} If  $E=E_{\lambda_1,\lambda_2}$, $\lambda_1\lambda_2\neq 0$ then 
$\dim \Ext^1_{\HH}(E, \II(\pi))=1$.
\end{lem}
\begin{proof} Applying $\Hom_{\HH}(\ast, \II(\pi))$ to \eqref{stex} gives an exact sequence
\begin{equation}\label{stex2}
  \begin{split}
 \Hom_{\HH}(\II(\pi), \II(\pi))\hookrightarrow &\Ext^1_{\HH}(\II(\pi), \II(\pi))\\ &\rightarrow  
\Ext^1_{\HH}(E, \II(\pi))\rightarrow \Ext^1_{\HH}(\II(\pi), \II(\pi))
\end{split}
\end{equation}
Hence, $\dim \Ext^1_{\HH}(E, \II(\pi))= 1 +\dim \Upsilon$, where  $\Upsilon$ is the image of the last arrow in  \eqref{stex2}. 
Yoneda's interpretation of $\Ext$ says that 
$\Upsilon\neq 0$ is equivalent to the following commutative diagram of $\HH$-modules:
\begin{displaymath}
\xymatrix{0\ar[r]& \II(\pi)\ar[r]\ar[d]^{=}& A\ar[r]\ar@{^(->}[d]& \II(\pi)\ar[r]\ar@{^(->}[d]& 0 \\
0\ar[r]& \II(\pi)\ar[r]& B \ar[r]& E \ar[r]& 0}
\end{displaymath}
with $A$ non-split. Then $A\cong E_{\mu_1, \mu_2}$ for some $\mu_1, \mu_2\in \Fbar$. The condition  
$v T_{n_s}^2=0$ for all $v\in B$ is equivalent to  $\mu_1\lambda_2=0$ and $\mu_2\lambda_1=0$. 
Since $\lambda_1\lambda_2\neq 0$ we obtain $\mu_1=\mu_2=0$ and hence a contradiction to a non-split $A$.  
\end{proof}

\section{Main result}

Let $\pi$ an irreducible representation with a central character $\zeta$.  A construction of \cite[\S6]{coeff}, 
\cite[\S9]{bp} gives an injection $\pi\hookrightarrow \Omega$, where $\Omega$ is in $\Rep_{G,\zeta}$
and $\Omega|_{K}$ is an injective envelope of $\soc_K \pi$ in $\Rep_{K,\zeta}$.

\begin{lem}\label{Opi} If $\pi\cong \pi(r, 0, \eta)$ with $0<r<p-1$ then $\Omega^{I_1}\cong E_{\lambda_1, \lambda_2}$
with  $\lambda_1\lambda_2\neq 0$. Otherwise, $\Omega^{I_1}\cong \pi^{I_1}$.
\end{lem} 
\begin{proof}
 Let $\sigma$ be an irreducible smooth representation of $K$ and $\Inj \sigma$ injective 
envelope of $\sigma$ in $\Rep_{K, \zeta}$. If $\sigma=\chi\circ \det$ or $\sigma\cong \St \otimes \chi\circ \det$ 
then $\dim (\Inj \sigma)^{I_1}=\dim \sigma^{I_1}=1$ and $\dim (\Inj \sigma)^{I_1}=2$ otherwise, \cite[6.4.1, \S4.1]{coeff}.
If $\pi$ is either a character, special series, a twist of unramified series or $\pi\cong \pi(0,0, \eta)$ 
then $\soc_K \pi$ is a direct summand of $(\Eins\oplus \St)\otimes \chi\circ \det$. Hence, 
$$\Omega^{I_1}=(\soc_K \Omega)^{I_1}=(\soc_K \pi)^{I_1} \subseteq \pi^{I_1}\subseteq \Omega^{I_1}$$  
and so  $\pi^{I_1}\cong \Omega^{I_1}$. If $\pi$ is  a tamely ramified principal series, which is not 
a twist of unramified principal series, then $\dim \pi^{I_1}=2$ and $\soc_K \pi$ is irreducible, so 
$\dim \Omega^{I_1}=2$. Finally,  if $\pi\cong \pi(r, 0, \eta)$ with $0< r< p-1$ then  it follows from 
\cite[6.4.5]{coeff} that $\Omega^{I_1}\cong E_{\lambda_1, \lambda_2}$ with $\lambda_1\lambda_2\neq 0$.
\end{proof} 

\begin{prop}\label{chop00} Let $\pi, \tau$ be irreducible representations of $G$ with a central character, 
and let $\zeta$ be the  central character of $\pi$. 
Suppose that $\Ext^1_G(\tau,\pi)\neq 0$. If 
\begin{equation}\label{mapdelta} 
\Ext^1_{G, \zeta}(\tau, \pi)\rightarrow \Hom_{\HH}(\II(\tau),\RR^1\II(\pi))
\end{equation} 
is not surjective then $\tau\cong \pi\cong \pi(r,0, \eta)$  with $0<r<p-1$. 
\end{prop}   
\begin{proof} We note that Proposition \ref{centre} implies that $\zeta$ is the  central character of $\tau$. Since  $\Omega|_K$ is 
an injective object in $\Rep_{K, \zeta}$, $\Omega|_{I_1}$ is an injective object in $\Rep_{I_1,\zeta}$. Hence, $\RR^1\II(\Omega)=0$ and we 
have an exact sequence: 
\begin{equation}
 0\rightarrow \II(\pi)\rightarrow \II(\Omega)\rightarrow \II(\Omega/\pi)\rightarrow \RR^1\II(\pi)\rightarrow 0.
\end{equation}
Assume  $\pi\cong \pi(r, 0,\eta)$, $0<r<p-1$. Let $\partial\in  \Hom_{\HH}(\II(\tau),\mathbb R^1\II(\pi))$ be non-zero.
Suppose that $\tau\not \cong \pi$ then  $\Ext^1_{\HH}(\II(\tau),\II(\pi))=0$,  \cite[6.5]{bp},
Lemma \ref{Opi} implies  $\II(\Omega)/\II(\pi)\cong \II(\pi)$. So  we have a  surjection
\begin{equation} \label{chop000}
\Hom_{\HH}(\II(\tau), \II(\Omega/\pi))\twoheadrightarrow \Hom_{\HH}(\II(\tau), \RR^1\II(\pi)).
\end{equation}
 Further, we have an isomorphism
\begin{equation}\label{aberhallo}
\Hom_G(\tau,\Omega/\pi)\cong \Hom_{\HH}(\II(\tau), \II(\Omega/\pi)).
\end{equation}
Choose $\psi\in \Hom_{G}(\tau, \Omega/\pi)$ mapping to $\partial$ under the composition of \eqref{aberhallo} and \eqref{chop000}.
Since $\tau$ is irreducible, by pulling back the image of $\psi$ we obtain an extension
$0\rightarrow \pi\rightarrow E_{\psi}\rightarrow \tau\rightarrow 0$ 
inside of $\Omega$. By construction, 
\eqref{mapdelta} maps the class  of this extension to $\partial$. 

If $\pi\not \cong \pi(r, 0,\eta)$ with $0<r<p-1$ then Lemma \ref{Opi} says that $\II(\Omega/\pi)\cong \RR^1\II(\pi)$ 
and arguing as above we get that \eqref{mapdelta} is surjective.
\end{proof}

\begin{cor}\label{chop1} Let $\pi$, $\tau$ be irreducible representations of $G$ with a central character, and 
suppose that $\pi$ is supersingular with a central character $\zeta$. 
If $\Ext^1_G(\tau,\pi)\neq 0$ and $\tau\not\cong \pi$ then 
$$\Ext^1_{G}(\tau,\pi)\cong \Hom_{\HH}(\II(\tau),\mathbb R^1\II(\pi)).$$
\end{cor}   
\begin{proof}  Proposition \ref{centre} implies that the central character of $\tau$ is $\zeta$ and 
$\Ext^1_G(\tau, \pi)\cong \Ext^1_{G, \zeta}(\tau, \pi)$. 
By \cite[Cor.6.5]{bp}, $\Ext^1_{\HH}(\II(\tau),\II(\pi))=0$. The assertion follows from Propositions \ref{5term}, \ref{chop00}.
\end{proof}

\begin{lem}\label{obv} Let $\pi$ and $\tau$ be supersingular representations of $G$ with the same central character. Suppose that 
$\pi^{I_1}\cong \tau^{I_1}$ as $H$-rep\-re\-sen\-tations then $\pi\cong \tau$.
\end{lem}
\begin{proof} It follows from the explicit description of supersingular modules $M(r,0,\eta)$ of $\HH$ in \S\ref{hecke} or \cite[Def.2.1.2]{coeff} that 
$\II(\tau)\cong\II(\pi)$ as $\HH$-modules. Hence, $\tau \cong \TT \II(\tau)\cong \TT\II(\pi)\cong \pi$.
\end{proof}

\begin{prop}\label{R1regp5} Let $\pi=\pi(r,0,\eta)$ with $0<r<p-1$, and let $\zeta$ be the central character of $\pi$. Assume that $p\ge 5$ then 
$\RR\II^1(\pi)\cong \II(\pi)\oplus \II(\pi)$.
\end{prop}
\begin{proof} Corollary \ref{three} implies that we have an isomorphism of $\HH^+$-modules $\RR^1\II(\pi)\cong \RR^1\II(\pi_{\sigma})\oplus \RR^1\II(\pi_{\tsigma})$.
Let $v\in \RR^1\II(\pi_{\sigma})$ it follows from Theorem \ref{main} that $ve_{\chi}=v$. Since $0<r<p-1$ we have $\chi\neq \chi^s$ and so $ve_{\chi^s}=0$. 
Since $T_{n_s}\in \HH^+$, $vT_{n_s}\in \RR^1\II(\pi_{\sigma})$ and hence $vT_{n_s}= vT_{n_s}e_{\chi}=ve_{\chi^s}T_{n_s}=0$. So $T_{n_s}$ kills $\RR^1\II(\pi_{\sigma})$ 
and by symmetry it also kills $\RR^1\II(\pi_{\tsigma})$. Theorem  \ref{main} gives $\dim \RR^1\II(\pi_{\sigma})=2$. If we chose a basis $\{v, w\}$ of 
$\RR^1\II(\pi_{\sigma})$ then $\{vT_{\Pi}, w T_{\Pi}\}$ is a basis of $\RR^1\II(\pi_{\tsigma})$. And it follows from the explicit description of 
$M(r,0,\eta)$ in \S\ref{hecke} that $\langle v, vT_{\Pi}\rangle$ is stable under the action of $\HH$ and is isomorphic to $M(r,0,\eta)$.
\end{proof}

\begin{prop}\label{R1regp3} Let $\pi$ and $\zeta$ be as in Proposition \ref{R1regp5} and let $\tau$ be an irreducible representation of $G$
with a central character $\zeta$. Assume $p>2$ and $\tau\not \cong \pi$ then $\Hom_{\HH}(\II(\tau), \RR^1\II(\pi))=0$.
\end{prop}
\begin{proof} Assume that $\Hom_{\HH}(\II(\tau), \RR^1\II(\pi))\neq 0$ if $p\ge 5$ then Proposition \ref{R1regp5} implies that $\II(\tau)\cong \II(\pi)$, 
and hence  $\tau\cong \pi$. Assume that $p=3$ then the assumption $0<r<p-1$ forces $r=1$. Corollary \ref{dimH1} implies that  
$\tau^{I_1}\cong \chi\oplus \chi^s$ as 
an $H$-representation, where $\chi$ is as in \eqref{explicitchi}. It follows from Lemma \ref{obv} that $\tau$ cannot be supersingular. Since $\chi\neq \chi^s$ 
we get that $\tau$ is a principal series representation. Corollary \ref{chop1} implies that $\Ext^1_G(\tau,\pi)\neq 0$. Let 
$\eta$ be one of the characters  $\omega\circ \det$, $\mu_{-1}\circ \det$, $\omega\mu_{-1}\circ \det$. Since $p=3$ and $r=1$, \eqref{intertwine} gives 
$\pi\cong \pi\otimes\eta$. Twisting by $\eta$ gives $\Ext^1_G(\tau\otimes \eta, \pi)\neq 0$, and hence $\Hom_{\HH}(\II(\tau\otimes \eta), \RR^1\II(\pi))\neq 0$.
Since $p>2$ \cite[Thm 34, Cor 36]{bl} imply that 
$\tau\not \cong \tau\otimes \eta$ and so $\II(\tau)\not \cong \II(\tau\otimes \eta)$ as $\HH$-modules. This implies that $\dim \RR^1\II(\pi)$ is at least 
$4\times 2=8$,
which contradicts Corollary \ref{dimH1}. 
\end{proof}

\begin{thm}\label{result} Assume that $p>2$ and let $\tau$ and $\pi$ be irreducible smooth representations of $G$ admitting a central character. 
Suppose that $\pi$ is supersingular and $\Ext^1_G(\tau, \pi)\neq 0$ then $\tau\cong \pi$.
\end{thm}
\begin{proof} If $0< r<p-1$ the assertion follows from Corollary \ref{chop1} and Proposition \ref{R1regp3}. Suppose that $r\in \{0,p-1\}$.
Let $\mathfrak I$ be the image of $\Ext^1_{G,\zeta}(\pi,\pi)\rightarrow \Hom_{\HH}(\II(\pi), \RR^1\II(\pi))$.  
Then it follows from Propositions \ref{CK},  \ref{5term}  and Lemma \ref{extH} that $\dim \mathfrak I\ge 3-1=2$. Hence, 
$\II(\pi)\oplus \II(\pi)$ is a submodule of $\RR^1\II(\pi)$. By forgetting the action of $\HH$ we obtain an isomorphism  of vector spaces
$\RR^1\II(\pi)\cong H^1(I_1/Z_1, \pi)$. Corollary \ref{dimH1} implies that $\dim \RR^1\II(\pi)=4$. Since $\dim \II(\pi)=2$ we obtain 
\begin{equation}\label{derived0}
\RR^1\II(\pi)\cong \II(\pi)\oplus \II(\pi).
\end{equation} 
Corollary  \ref{chop1} implies the result. 
\end{proof}
\begin{remar} We note that the proof in the regular case $0<r<p-1$ is purely representation theoretic and makes no use of 
Colmez's functor. The Iwahori case $r\in\{0, p-1\}$ could also be done representation theoretically. One needs to work out
the action of $\HH$ on $H^1(I_1/Z_1, \pi)$. This can be done, but it is not so pleasant, in particular $p=3$ requires extra arguments.
\end{remar}

\begin{lem}\label{pupsi} Let $\pi\cong \pi(r,0, \eta)$ with $0<r<p-1$, then 
$$\dim \II(\Omega/\pi) \centerdot e_{\chi}T_{n_s}\ge 1, \quad \dim \II(\Omega/\pi) \centerdot e_{\chi^s}T_{n_s}\ge 1. $$
\end{lem}
\begin{proof} We have an exact sequence of $K$-representations:
\begin{equation}
0\rightarrow \pi^{K_1}\rightarrow \Omega^{K_1}\rightarrow (\Omega/\pi)^{K_1}
\end{equation}
Since $\Omega|_{K}\cong \Inj \sigma \oplus \Inj \tsigma$, we have $\Omega^{K_1}\cong \inj \sigma \oplus \inj \tsigma$, 
where $\inj$ denotes an injective envelope in the category $\Rep_{K/K_1}$, \cite[6.2.4]{coeff}. In \cite[20.1, \S16]{bp} \
we have determined the $K$-representation $\pi^{K_1}\cong \pi^{K_1}_{\sigma}\oplus \pi^{K_1}_{\tsigma}$. It follows from 
the description and \cite[3.4, 3.5]{bp} that $\pi^{K_1}_{\sigma}$ is isomorphic to the kernel 
of $\inj \sigma \twoheadrightarrow \Indu{I}{K}{\chi}$. Hence, $(\Omega/\pi)^{K_1}$ contains 
$\Indu{I}{K}{\chi}\oplus \Indu{I}{K}{\chi^s}$ as a subobject and so $(\Omega/\pi)^{I_1}$ contains 
$V:=(\Indu{I}{K}{\chi}\oplus \Indu{I}{K}{\chi^s})^{I_1}.$
Moreover, $V$ is stable under the action of  $T_{n_s}$, and  $\dim V e_{\chi} T_{n_s}=\dim V e_{\chi^s} T_{n_s}=1$, 
\cite[3.1.11]{coeff}. This yields the claim.
\end{proof} 

\begin{prop}\label{gotchabound}  Let $\pi\cong \pi(r,0, \eta)$ with $0<r<p-1$. If $p\ge 5$ then 
\begin{equation}
\dim \Ext^1_{K/Z_1}(\sigma, \pi)\le 2, \quad \dim \Ext^1_{K/Z_1}(\tsigma, \pi)\le 2.
\end{equation}
 If $p=3$ then 
\begin{equation}
\dim \Ext^1_{K/Z_1}(\sigma, \pi)\le 3, \quad \dim \Ext^1_{K/Z_1}(\tsigma, \pi)\le 3.
\end{equation}
\end{prop} 
\begin{proof} We have $\Hom_{K/Z_1}(\sigma, \pi)\cong  \Hom_{K/Z_1}(\sigma, \Omega)$, since by construction $\soc_K \Omega \cong \soc_K \pi$.
Moreover, since $\Omega|_K$ is injective in $\Rep_{K, \zeta}$ we have $\Ext^1_{K/Z_1}(\sigma, \Omega)=0$. Hence, 
\begin{equation}\label{gotcha1}
 \Hom_{K/Z_1}(\sigma, \Omega/\pi)\cong \Ext^1_{K/Z_1}(\sigma, \pi).
\end{equation}
It follows from \cite[4.1.5]{coeff} that if  $\kappa$ is any smooth $K$-representation then one has 
\begin{equation}\label{gotcha2}
 \Hom_{K/Z_1}(\sigma, \kappa)\cong \Ker ( \II(\kappa)e_{\chi} \overset{T_{n_s}}{\longrightarrow}  \II(\kappa)e_{\chi^s}).
\end{equation}
Now Lemma \ref{pupsi}, \eqref{gotcha1} and \eqref{gotcha2} imply that 
\begin{equation}\label{gotcha3}
 \dim \Ext^1_{K/Z_1}(\sigma, \pi)\le \dim \II(\Omega/\pi)e_{\chi}- 1= \dim \RR^1\II(\pi)e_{\chi}.
\end{equation}
It follows from Theorem \ref{main} that if $p\ge 5$ then $\dim  \RR^1\II(\pi)e_{\chi}=2$ and if $p=3$ then 
$\dim  \RR^1\II(\pi)e_{\chi}\le 3$. The same proof also works for  $\tsigma$.
\end{proof}

\begin{prop}\label{apchibound}  Let $\pi\cong \pi(r,0, \eta)$ with $0<r<p-1$. If $p\ge 5$ then 
\begin{equation}
\dim \Ext^1_{G, \zeta}(\pi, \pi)\le 3.
\end{equation}
 If $p=3$ then 
\begin{equation}
\dim \Ext^1_{G, \zeta}(\pi, \pi)\le 4.
\end{equation}
\end{prop} 
\begin{proof} Recall that   we have an exact sequence:
\begin{equation}\label{apchi1}
0\rightarrow \cIndu{KZ}{G}{\sigma}\overset{T}{\rightarrow} \cIndu{KZ}{G}{\sigma}\rightarrow \pi\rightarrow 0.
\end{equation}
Applying $\Hom_{G}(\ast, \pi)$ to \eqref{apchi1} gives an exact sequence
\begin{equation}\label{apchi2}
\Hom_G(\cIndu{KZ}{G}{\sigma}, \pi)\hookrightarrow \Ext^1_{G, \zeta}(\pi, \pi)\rightarrow \Ext^1_{G, \zeta}(\cIndu{KZ}{G}{\sigma}, \pi).
\end{equation}
We may think of this exact sequence first as Yoneda Exts in $\Rep_{G, \zeta}$, but since  $\Rep_{G, \zeta}$ has enough 
injectives Yoneda's $\Ext^n$ is isomorphic to $\RR^n \Hom\cong \Ext^n_{G, \zeta}$. For any $A$ in $\Rep_{G, \zeta}$ 
we have 
$$\Hom_G(\cIndu{KZ}{G}{\sigma}, A)\cong \Hom_{K/Z_1}(\sigma, \mathrm F A),$$ 
where $\mathrm F: \Rep_{G,\zeta}\rightarrow \Rep_{K, \zeta}$ is  the restriction. The functor $\mathrm F$ is exact and 
maps injectives to injectives, hence   
\begin{equation}\label{apchi3}
 \Ext^1_{G, \zeta}(\cIndu{KZ}{G}{\sigma}, A)\cong  \Ext^1_{K/Z_1}(\sigma, \mathrm F A). 
\end{equation}
Now \eqref{apchi2}, \eqref{apchi3} and Proposition \ref{gotchabound} give the assertion.
\end{proof} 
The same proof gives:
\begin{cor}\label{wanted} Let $n\ge 1$ and $\tau=\frac{\cIndu{KZ}{G}{\sigma}}{(T^n)}$ or  $\tau=\frac{\cIndu{KZ}{G}{\tsigma}}{(T^n)}$.
If $p\ge 5$ then $\dim \Ext^1_{G, \zeta}(\tau, \pi)\le 3$; if $p=3$ then  $\dim \Ext^1_{G, \zeta}(\tau, \pi)\le 4$.
\end{cor}

\begin{thm}\label{equality} Assume $p>2$ and $\pi\cong \pi(r, 0, \eta)$ supersingular. If  $(p,r)\neq (3,1)$ then 
$\dim \Ext^1_{G, \zeta}(\pi, \pi)=3$. 
\end{thm}
\begin{proof} Proposition \ref{CK} or \S\ref{app} gives $\dim \Ext^1_{G, \zeta}(\pi, \pi)\ge 3$. If $0<r<p-1$ then 
equality follows from Proposition \ref{apchibound}. If $r=0$ or $r=p-1$ then $\Ext^2_{\HH}(\II(\pi), \II(\pi))=0$ 
 and $\Ext^1_{\HH}(\II(\pi),\II(\pi))$ is $1$-dimensional by Lemma \ref{extH}. Hence, \eqref{derived0} and \eqref{5T} give 
$\dim \Ext^1_{G, \zeta}(\pi, \pi)=3$. 
\end{proof}

For future use we record the following: 
\begin{prop} Assume $p>2$ and $\pi\cong \pi(r, 0, \eta)$ supersingular. Let 
$0\rightarrow \II(\pi)\rightarrow E \rightarrow \II(\pi)\rightarrow 0$ be a non-split 
extension of $\HH$-modules.  If  $(p,r)\neq (3,1)$ then 
$\dim \Ext^1_{G, \zeta}(\TT(E), \pi)\le 3$.
\end{prop}
\begin{proof} If $(p,r)\neq (3,1)$ then we have $\RR^1\II(\pi)\cong \II(\pi)\oplus \II(\pi)$ and 
so $\dim \Hom_{\HH}(E, \RR^1\II(\pi))=2$. So if $\dim \Ext^1_{\HH}(E, \II(\pi))\le 1$ then 
\eqref{5T} allows us to conclude. If $r=0$ or $r=p-1$ the latter may be deduced from Lemma \ref{extH}. If $0<r<p-1$  and 
$E\cong E_{\lambda_1, \lambda_2}$ with $\lambda_1\lambda_2\neq 0$ then this is given by Lemma \ref{kuku}. If 
$\lambda_1\lambda_2=0$ then $\TT(E)\cong \frac{\cIndu{KZ}{G}{\sigma}}{(T^2)}$ or  $\tau=\frac{\cIndu{KZ}{G}{\tsigma}}{(T^2)}$
and the assertion is given by Corollary \ref{wanted}.
\end{proof}

\section{Non-supersingular representations}\label{non} 
We  compute $\Ext^1_{G,\zeta}(\tau,\pi)$, 
when $\pi$ is the Steinberg representation of $G$ or a character and $\tau$ is an irreducible representation of $G$ under the assumption $p>2$. The 
results of this paper combined with  \cite{bp} give all the extensions between irreducible representations of $G$, when $p>2$. We record this below. A lot of 
cases have been worked out by different methods by Colmez \cite{col1} and Emerton \cite{em}. The new results of this section 
are determination of $\RR^1\II(\Sp)$, where $\Sp$ is the Steinberg representation, and showing that if $\eta: G\rightarrow \Fbar^{\times}$ 
is a smooth character of order $2$ then $\Ext^1_G(\eta, \Sp)=0$.

\begin{prop}\label{H1Sp} Assume $p>2$ and let $\psi:H\rightarrow \Fbar^{\times}$ be a character. 
Suppose that $\Ext^1_{I/Z_1}(\psi,\Sp)\neq 0$ then $\psi=\Eins$ the trivial 
character. Moreover, $\dim \Ext^1_{I/Z_1}(\Eins, \Sp)=2$.
\end{prop}
\begin{proof} It follows from \eqref{induced} that $\pi(p-1,1)\cong \Indu{P}{G}{\Eins}$. By restricting \eqref{zero} to $I$ we obtain 
an exact sequence of $I$-representations:
\begin{equation}\label{exseqSteinberg}
0\rightarrow \Eins \rightarrow \Indu{I\cap P^s}{I}{\Eins}\oplus \Indu{I\cap P}{I}{\Eins}\rightarrow \Sp\rightarrow 0. 
\end{equation}
If we set $M:= \Indu{I\cap P^s}{I}{\Eins}$ then $\Indu{I\cap P}{I}{\Eins}\cong M^{\Pi}$, and $M|_{H(I\cap U)}\cong \Indu{H}{H(I\cap U)}{\Eins}$ is an injective envelope 
of $\Eins$ in $\Rep_{H(I\cap U)}$. So \eqref{exseqSteinberg} is an analog of Theorem \ref{exseq}. The proof of Theorem \ref{main} goes through without any changes. For $p=3$ we 
note that $M^{K_1}\cong \Indu{HK_1}{I}{\Eins}$ and hence $M$ satisfies the assumptions of Lemma \ref{done}.
\end{proof}

Let $\omega:\Qp^{\times}\rightarrow \Fbar^{\times}$ be a character, such that $\omega(p)=1$ and $\omega|_{\Zp^{\times}}$ is the reduction map
composed with the canonical embedding.

\begin{prop}\label{Rst} Assume $p>2$ then $\RR^1\II(\Eins)\cong M(p-3,1, \omega)$ and $\RR^1\II(\Sp)\cong M(p-1,1)$,
\end{prop}    
\begin{proof} Recall  \eqref{zero} gives an exact sequence
\begin{equation}\label{zerobis}
 0 \rightarrow \Eins\rightarrow \pi(p-1,1)\rightarrow \Sp\rightarrow 0. 
\end{equation}
Applying $\II$ to \eqref{zerobis} we get an exact sequence: 
$$0\rightarrow \RR^1\II(\Eins) \rightarrow \RR^1\II(\pi(p-1,1))\rightarrow \RR^1\II(\Sp).$$
Now \cite[Thm.7.16]{bp} asserts that $\RR^1\II(\pi(p-1,1))\cong M(p-3,1,\omega)\oplus M(p-1,1)$.
Now $H$ acts on $\RR^1\II(\Eins)$ and $\RR^1\II(\pi(p-1,1))$ via $h\mapsto T_{h^{-1}}$. It follows from 
Definition \ref{Mrlambdaeta}  that 
$$M(p-1,1)\cong \Eins\oplus \Eins, \quad M(p-3,1, \omega)\cong \alpha\oplus \alpha^{-1}$$
as $H$-representations.  Propositions \ref{homi}, \ref{extI} imply that  
$$\RR^1\II(\Eins)\cong H^1(I_1/Z_1,\Eins)\cong \Hom(I_1/Z_1,\Fbar)=\langle \kappa^u,\kappa^l\rangle\cong \alpha\oplus\alpha^{-1}$$
as $H$-representations. Since $p>2$ we get $\RR^1\II(\Eins)\cong M(p-3,1, \omega)$. Then $M(p-1,1)$ is 
a $2$-dimensional submodule of $\RR^1\II(\Sp)$. However, Proposition \ref{H1Sp} implies that $\RR^1\II(\Sp)$
is $2$-dimensional, so the injection is an isomorphism. 
\end{proof} 

\begin{lem}\label{extmod1Sp} Let $M$ be an irreducible $\HH$-module. If $\Ext^1_{\HH}(M,\II(\Eins))$ or $\Ext^1_{\HH}(M,\II(\Sp))$
is non-zero then $M\in \{\II(\Eins), \II(\Sp)\}$. Moreover, 
$$\dim \Ext^1_{\HH}(\II(\Eins),\II(\Sp))=\dim \Ext^1_{\HH}(\II(\Sp),\II(\Eins))=1.$$
If $p>2$ then $\Ext^1_{\HH}(\II(\Eins),\II(\Eins))$ and $\Ext^1_{\HH}(\II(\Sp),\II(\Sp))$ are zero, and if $p=2$
then  both spaces are $1$-dimensional.
\end{lem}  
\begin{proof} Recall that \eqref{p-1} gives an exact sequence:
\begin{equation}\label{p-1bis}
 0 \rightarrow \Sp \rightarrow \pi(0,1)\rightarrow \Eins\rightarrow 0. 
\end{equation}
Applying $\II$ we obtain an exact sequence:
\begin{equation}\label{Sptriv}
0\rightarrow \II(\Sp)\rightarrow M(0,1)\rightarrow \II(\Eins)\rightarrow 0.
\end{equation}
If $\Ext^1_{\HH}(M,\II(\Sp))\neq 0$ and $M\not \cong \II(\Eins)$ then from \eqref{Sptriv} 
we obtain  that $\Ext^1_{\HH}(N,M(0,1))\neq 0$, and \cite[Cor.6.5]{bp}
implies that $M$ is either a subquotient of $M(0,1)$ or a subquotient of $M(p-1,1)$. Hence $M\cong \II(\Sp)$. Using \eqref{zerobis} one 
can deal in the same way with $\Ext^1_{\HH}(N,\II(\Eins))$. Since $\II(\Eins)$ and $\II(\Sp)$ are one dimensional, one can verify 
the rest by hand using the description of $\HH$ in terms of generators and relations given in  \eqref{relTH}.
\end{proof}

Let $\pi$ and $\tau$ be irreducible representations of $G$ admitting the same central character $\zeta$. Assume that $\pi$ is not supersingular.
When $p>2$ for given $\pi$ we are going to list all $\tau$ such that $\Ext^1_{G,\zeta}(\tau,\pi)\neq 0$. If one is interested in 
$\Ext^1_G(\tau,\pi)$ then this can be deduced from Proposition \ref{centre}. If $\eta:G\rightarrow \Fbar^{\times}$ is a smooth character then 
$\Ext^1_{G,\zeta}(\tau\otimes \eta, \pi\otimes\eta)\cong \Ext^1_{G,\zeta}(\tau, \pi)$. Hence, we may assume that 
$\pi$ is $\Eins$, $\Sp$ or $\pi(r,\lambda)$ with $\lambda\neq 0$ and $(r,\lambda)\neq(0,\pm 1)$, $(r,\lambda)\neq (p-1,\pm 1)$. Recall  
if $\lambda\neq 0$ and $(r,\lambda)\neq (0,\pm 1)$ then \cite[Thm.30]{bl} asserts that 
\begin{equation}\label{inducedbis}
 \pi(r,\lambda)\cong \Indu{P}{G}{\mu_{\lambda^{-1}}\otimes \mu_{\lambda}\omega^r}.
\end{equation}
It follows from \eqref{inducedbis} that if $\lambda\neq \pm 1$ then $\pi(0,\lambda)\cong \pi(p-1,\lambda)$. Hence, we may assume that $1\le r \le p-1$.  
Propositions \ref{5term} and  \ref{chop00} gives us an exact sequence:
\begin{equation}\label{nearly}
\Ext^1_{\HH}(\II(\tau), \II(\pi))\hookrightarrow\Ext^1_{G,\zeta}(\tau,\pi)\twoheadrightarrow \Hom_{\HH}(\II(\tau),\RR^1\II(\pi)).
\end{equation}

\begin{thm} Let $\pi$, $\tau$ and $\zeta$ be as above. Assume that $p>2$ and $\Ext^1_{G,\zeta}(\tau,\pi)\neq 0$. Let $d$ be the 
dimension of $\Ext^1_{G,\zeta}(\tau,\pi)$.
\begin{itemize} 
\item[(i)] if $\pi\cong \Eins$ then one of the following holds: 
             \begin{itemize} 
                \item[(a)] $\tau\cong \Sp$, and $d=1$;
                \item[(b)]  $p\ge 5$, $\tau\cong \pi(p-3,1,\omega)\cong \Indu{P}{G}{\omega\otimes \omega^{-1}}$ and $d=1$;
                \item[(c)] $p=3$, $\tau\cong \Sp\otimes \omega\circ \det$ and $d=1$;
              \end{itemize}  
\item[(ii)] if $\pi\cong \Sp$ then $\tau\cong \Eins$ and $d=2$;
\end{itemize}
\end{thm}
\begin{proof} This follows from \eqref{nearly}, Lemma \ref{extmod1Sp} and Proposition \ref{Rst}. We note that if 
$p=3$ then   $\pi(p-3,1,\omega)$ is reducible, but has a unique irreducible subobject isomorphic to $\Sp\otimes \omega\circ \det$.
\end{proof}
For the sake of completeness we also deal with $\Ext^1_{G,\zeta}(\tau,\pi)$ when $\pi$ is irreducible principal series. We deduce the results 
from \cite[\S8]{bp}, but they are also contained in \cite{col1} and \cite{em}.

\begin{thm} Let $\pi$, $\tau$ and $\zeta$ be as above. Assume that $p>2$, $\pi\cong \pi(r,\lambda)$ with $1\le r \le p-1$, $\lambda\in \Fbar^{\times}$ and 
$(r,\lambda)\neq (p-1,\pm 1)$. Then 
$$\Ext^1_{G,\zeta}(\pi(r,\lambda),\pi(r,\lambda))\cong \Hom(\Qp^{\times}, \Fbar).$$
In particular, $\dim \Ext^1_{G,\zeta}(\pi(r,\lambda),\pi(r,\lambda))=2$. Moreover, suppose that $\tau\not \cong \pi$ and $\Ext^1_{G,\zeta}(\tau,\pi)\neq 0$. Let $d$ be the 
dimension of $\Ext^1_{G,\zeta}(\tau,\pi)$ then one of the following holds:
\begin{itemize} 
\item[(i)] if $(r,\lambda)=(p-2,\pm 1)$ then such $\tau$ does not exist;
\item[(ii)] if $(r,\lambda)=(p-3, \pm 1)$ (hence $p\ge 5$) then $\tau\cong \Sp\otimes \omega^{-1}\mu_{\pm 1}\circ \det$ and $d=1$; 
\item[(iii)] otherwise, $\tau\cong \pi(s,\lambda^{-1}, \omega^{r+1})$, where $0\le s\le p-2$ and $s\equiv p-3-r\pmod{p-1}$, and $d=1$.
\end{itemize}
\end{thm}
\begin{remar} Note that if $\pi=\pi(r, \lambda)$ is as in (iii) and we write $\pi\cong \Indu{P}{G}{\psi_1\otimes\psi_2\omega^{-1}}$, then it follows from \eqref{inducedbis}
that $\pi(s,\lambda^{-1}, \omega^{r+1})\cong\Indu{P}{G}{\psi_2\otimes\psi_1\omega^{-1}}$.
\end{remar}
\begin{proof} The first assertion follows from \cite[Cor.8.2]{bp}. Assume that $\tau\not \cong \pi$ then it follows from \cite[Cor.6.5, 6.6, 6.7]{bp} 
that $\Ext^1_{\HH}(\II(\tau),\II(\pi))=0$. Hence, \eqref{nearly} implies that $\Ext^1_{G,\zeta}(\tau,\pi)\cong \Hom_{\HH}(\II(\tau),\RR^1\II(\pi))$. 
The assertions (i),(ii) and (iii) follow from \cite[Thm.7.16]{bp}, where $\RR^1\II(\pi)$ is determined. The difference between (ii) and (iii) is 
accounted for by the fact that if $r=p-3$ then $s=0$ and if $\lambda=\pm 1$ then  $\pi(s,\lambda^{-1},\omega^{r+1})\cong 
\pi(0,\pm 1, \omega^{p-2})$, which is reducible, but has a unique irreducible submodule isomorphic to $\Sp\otimes \omega^{-1}\mu_{\pm 1}\circ \det$.
\end{proof}

\appendix
\section{Lower bound on $\dim \Ext^1_G(\pi, \pi)$}\label{app}
Let $\FF$ be a finite field of characteristic $p>2$ and $W(\FF)$ the ring of Witt vectors. 
Let $0\le r\le p-1$ be an integer and set 
$$\pi(r):= \frac{\cIndu{KZ}{G}{\Sym^r \FF^2}}{(T)}.$$
We note that the endomorphism $T$ is defined over $\FF$, see \cite[Prop 1]{bl}. In this section, we bound the dimension
of $\Ext^1_{G}(\pi(r), \pi(r))$ from below, using the ideas of Colmez and Kisin. Let $L$ be a finite extension of $W(\FF)[1/p]$ and 
$\OO$ the ring of integers in $L$. Let $\GG_{\Qp}$ be the absolute Galois group of $\Qp$. Let $\Rep_{\OO} G$ be the category 
of $\OO[G]$-modules of finite length, with the central character, and such that the action of $G$ is continuous for the discrete 
topology. Let $\Rep_{\OO}\GG_{\Qp}$ be the category of $\OO[\GG_{\Qp}]$-modules of finite length, such that the action of $\GG_{\Qp}$ 
is continuous for the discrete topology. Colmez in \cite{col1} has defined an exact functor 
$$\VV: \Rep_{\OO} G\rightarrow \Rep_{\OO} \GG_{\Qp}.$$   
Set $\rho(r):=\VV(\pi(r))$, then $\rho(r)$ is an absolutely irreducible $2$-di\-men\-sio\-nal $\FF$-representation of $\GG_{\Qp}$, uniquely 
determined by the following: $\det \rho =\omega^{r+1}$; the restriction of $\rho$ to inertia is isomorphic to $\omega_2^{r+1}\oplus \omega_2^{p(r+1)}$,
where $\omega_2$ is the fundamental character of Serre of niveau $2$. 
In the notation of \cite{b2}, $\rho(r)=\ind \omega_2^{r+1}$. We note that 
since, $\pi(r)$ and $\rho(r)$ are absolutely irreducible, the functor $\VV$ induces an isomorphism:
\begin{equation}\label{isohom}
\Hom_G(\pi(r),\pi(r))\cong \Hom_{\GG_{\Qp}}(\rho(r), \rho(r))\cong \FF.
\end{equation}
Let $\eta: \GG_{\Qp}\rightarrow \OO^{\times}$ be a crystalline character lifting $\zeta:=\omega^r$ the central character of $\pi(r)$. We consider
$\eta$ as a character of the centre of $G$, $Z(G)\cong \Qp^{\times}$ via the class field theory. To simplify the notation we set $\pi:=\pi(r)$
and $\rho:=\rho(r)$. Let $\Rep_{\OO}^{\pi, \eta} G$ be the full subcategory of $\Rep_{\OO} G$, such that $\tau$ is an object in $\Rep_{\OO}^{\pi, \eta} G$
if and only if the central character of $\tau$ is equal  to (the image of) $\eta$, and the irreducible subquotients of $\tau$ are isomorphic to 
$\pi$. We note that $\Rep_{\OO}^{\pi, \eta} G$ is abelian.

For $\tau$ and $\kappa$ in $\Rep_{\OO}^{\pi, \eta} G$ we let $\Ext^1_G(\kappa, \tau)$ be the Yoneda $\Ext^1$ in 
$\Rep_{\OO}^{\pi, \eta} G$, so an element of 
$\Ext^1_{G}(\kappa, \tau)$ can be viewed as an equivalence class of an  exact sequence 
\begin{equation}\label{explain}
0\rightarrow \tau\rightarrow E\rightarrow \kappa\rightarrow 0,
\end{equation}
where $E$ lies in $\Rep_{\OO}^{\pi, \eta} G$. Applying $\VV$ to \eqref{explain} 
we get an exact sequence $0\rightarrow \VV(\tau)\rightarrow \VV(E)\rightarrow \VV(\kappa) \rightarrow 0$. Hence, a map 
\begin{equation}\label{VExt}
\Ext^1_{G}(\kappa, \tau)\rightarrow \Ext^1_{\GG_{\Qp}}(\VV(\kappa), \VV(\tau)).
\end{equation}
A theorem of Colmez \cite[VII.5.3]{col1} asserts that \eqref{VExt} is injective, when $\tau=\kappa=\pi$.  

\begin{lem}\label{Ext1inj} Let $\tau$ and $\kappa$ be in $\Rep_{\OO}^{\pi, \eta} G$ then $\VV$ induces an isomorphism, and an injection respectively: 
$$\Hom_{G}(\kappa, \tau)\cong \Hom_{\GG_{\Qp}}(\VV(\kappa), \VV(\tau)),$$
$$\Ext^1_G(\kappa, \tau)\hookrightarrow \Ext^1_{\GG_{\Qp}}(\VV(\kappa), \VV(\tau)).$$
\end{lem}
\begin{proof} We may assume that $\tau\neq 0$ and $\kappa\neq 0$. We argue by induction on 
$\ell(\tau)+\ell(\kappa)$, where $\ell$ is the 
length as an $\OO[G]$-module. If $\ell(\tau)+\ell(\kappa)=2$ then $\tau\cong \kappa\cong \pi$ and
 the assertion about $\Ext^1$ is a Theorem of Colmez cited above, the assertion about $\Hom$ follows from \eqref{isohom}.
Assume that
$\ell(\tau)>1$ then we have an exact sequence: 
\begin{equation}\label{exactseq1}
0\rightarrow \tau'\rightarrow \tau\rightarrow \pi\rightarrow 0.
\end{equation} 
Since $\VV$ is exact we get an exact sequence:
\begin{equation}\label{exactseq2}
0\rightarrow \VV(\tau')\rightarrow \VV(\tau)\rightarrow \VV(\pi)\rightarrow 0.
\end{equation}
Applying $\Hom_G(\kappa, \centerdot)$ to \eqref{exactseq1} 
and $\Hom_{\GG_{\Qp}}(\VV(\kappa), \centerdot)$ to \eqref{exactseq2} we obtain two long exact sequences, and a map between them
induced by $\VV$. With the obvious notation we get a commutative diagram: 
\begin{displaymath}
\xymatrix{0\ar[r]& A^0\ar[r]\ar[d]^{\cong}& B^0\ar[r]\ar[d]& C^0\ar[r]\ar[d]^{\cong}& A^1\ar[r]\ar@{^(->}[d]& B^1\ar[r]\ar[d]& 
C^1\ar@{^(->}[d]\\
0\ar[r]& \mathcal{A}^0\ar[r]& \mathcal{B}^0\ar[r]& \mathcal{C}^0\ar[r]& \mathcal{A}^1\ar[r]& \mathcal{B}^1\ar[r]& \mathcal{C}^1.}
\end{displaymath}
The first and third vertical arrows are isomorphisms,  fourth and sixth injections by induction hypothesis. This implies that 
the second arrow is an isomorphism, and the fifth is an injection. Hence,
$$\Hom_{G}(\kappa, \tau)\cong \Hom_{\GG_{\Qp}}(\VV(\kappa), \VV(\tau)),$$ 
$$\Ext^1_G(\kappa, \tau)\hookrightarrow \Ext^1_{\GG_{\Qp}}(\VV(\kappa), \VV(\tau)).$$

If $\ell(\tau)=1$ and $\ell(\kappa)>1$ then one may argue similarly with 
$\Hom_G( \centerdot, \tau)$ and  $\Hom_{\GG_{\Qp}}(\centerdot, \VV(\tau))$. 
\end{proof}

From now on we assume that $(p,r)\neq (3,1)$.  Let $\Rep_{\OO}^{\pi,\eta}\GG_{\Qp}$ be the full 
subcategory of $\Rep_{\OO}\GG_{\Qp}$, with objects $\rho'$, such that there exists $\pi'$ in $\Rep_{\OO}^{\pi,\eta} G$ with 
$\rho'\cong \VV(\pi')$. Lemma \ref{Ext1inj} implies that $\VV$ induces an equivalence of categories between $\Rep_{\OO}^{\pi, \eta} G$ and 
$\Rep_{\OO}^{\pi,\eta} \GG_{\Qp}$. In particular, $\Rep_{\OO}^{\pi,\eta} \GG_{\Qp}$ is abelian. We define three deformation problems for $\rho$, closely
following Mazur \cite{mazur}. Let $D^{u}$ be the universal deformation problem; $D^{\omega\eta}$ the deformation problem with the
determinant condition, so that we consider the deformations with  determinant equal to $\omega \eta$, \cite[\S24]{mazur}; $D^{\pi,\eta}$ 
 a deformation problem with the categorical condition, so that we consider those deformations, which as representations 
of $\OO[\GG_{\Qp}]$ lie in $\Rep_{\OO}^{\pi,\eta} \GG_{\Qp}$, \cite[\S25]{mazur}, \cite{rama}. 
Since $\rho$ is absolutely irreducible, the functors
$D^{u}$, $D^{\omega\eta}$, $D^{\pi, \eta}$ are (pro-)representable by  complete local noetherian $\OO$-algebras $R^{u}$, $R^{\omega\eta}$, 
$R^{\pi, \eta}$ respectively. By the universality of $R^u$ we have surjections 
$R^u\twoheadrightarrow R^{\omega\eta}$ and $R^u\twoheadrightarrow R^{\pi, \eta}$. 

For $\rho'$ in $\Rep_{\FF} \gal$ we set $h^i(\rho'):=\dim_{\FF} H^i(\GG_{\Qp}, \rho')$. Let $V$ be 
the underlying vector space of $\rho$, the $\GG_{\Qp}$ acts by conjugation on $\End_{\FF} V$. We denote this representation by 
$\mathrm{Ad}(\rho)$, in particular $\mathrm{Ad}(\rho)\cong \rho\otimes\rho^*$. Local Tate duality gives 
$$h^2(\rho\otimes\rho^*)=h^0(\rho\otimes\rho^* \otimes \omega)=\dim \Hom_{\GG_{\Qp}}(\rho,\rho\otimes \omega).$$ 
Now \cite[Lem. 4.2.2]{b1} implies that  $\rho\cong \rho\otimes \omega$ if and only if $p=2$ or $(p,r)=(3,1)$. Since both 
cases are excluded here, we have $h^2(\mathrm{Ad}(\rho))=0$. Since $\rho$ is absolutely irreducible 
$h^0(\rho\otimes\rho^*)=1$. The local Euler characteristic gives:
 $$4=\dim \rho\otimes\rho^*=-h^0(\rho\otimes\rho^*)+h^1( \rho\otimes\rho^*)-h^2(\rho\otimes\rho^*)$$
and so $h^1(\mathrm{Ad}(\rho))=5$. Since $p>2$ the exact sequence of $\GG_{\Qp}$-representations:
$$0\rightarrow \mathrm{Ad}^0(\rho)\rightarrow \mathrm{Ad}(\rho)\overset{trace}{\rightarrow} \FF\rightarrow 0$$
splits. Hence $h^1(\mathrm{Ad}^0(\rho))=3$ and $h^2(\mathrm{Ad}^0(\rho))=0$. It follows from \cite{mazur2} that 
$R^u\cong \OO[[t_1,\ldots, t_5]]$ and $R^{\omega\eta}\cong \OO[[t_1,t_2, t_3]]$.

Inverting $p$ we get surjections $R^u[1/p]\twoheadrightarrow R^{\omega\eta}[1/p]$ and 
$R^u[1/p]\twoheadrightarrow R^{\pi, \eta}[1/p]$, 
and hence closed embeddings 
$$\Spec R^{\omega\eta}[1/p]\hookrightarrow \Spec R^u[1/p],\quad \Spec R^{\pi,\eta}[1/p]\hookrightarrow \Spec R^u[1/p].$$
Let $x\in \Spec R^{\omega\eta}[1/p]$ be a closed point with residue field $E$. Specializing at $x$ we 
obtain a continuous $2$-dimensional 
$E$-representation $V_x$ of $\GG_{\Qp}$. Suppose that $V_x$ is crystalline, and if $\lambda_1, \lambda_2$ are eigenvalues 
of $\varphi$ on $D_{crys} (V_x^*)$ then $\lambda_1\neq \lambda_2$ and $\lambda_1\neq \lambda_2 p^{\pm 1}$ then Berger-Breuil in  
\cite{bergerbreuil} associate to $V_x$  
a unitary $E$-Banach space representation $B_x$ of $G$. Choose a $G$-invariant norm $\|\centerdot\|$ on $B_x$ defining the topology and such that
$\|B_x\|\subseteq |E|$ and let $B^0_x$ be the unit ball with respect to $\|\centerdot\|$. Berger has shown in \cite{berger} that 
$B^0_x\otimes_{\oE} \FF\cong \pi$ as $G$-representations.  The constructions in \cite{bergerbreuil} and 
\cite{col1} are mutually inverse to one another. This means 
$$V_x\cong E\otimes_{\oE}\underset{\leftarrow}{\lim}\ \mathbf V(B_x^0/\varpi_E^n B_x^0).$$
Hence, every such $x$ also lies in $\Spec R^{\pi,\eta}[1/p]$. A Theorem of Kisin \cite[1.3.4]{kis1} asserts that the set
of crystalline  points, satisfying the conditions  above, is Zariski dense in $\Spec R^{\omega\eta}[1/p]$. 
Since $\Spec R^{\omega\eta}[1/p]$
and $\Spec R^{\pi, \eta}[1/p]$ are closed subsets of $\Spec R^u[1/p]$, we get that $\Spec R^{\omega\eta}[1/p]$ is 
contained in $\Spec R^{\pi, \eta}[1/p]$. Since $R^{\omega\eta}[1/p]$ is reduced we get a surjective homomorphism
$R^{\pi, \eta}[1/p]\twoheadrightarrow R^{\omega\eta}[1/p]$. Let $I$ be the kernel of $R^u\twoheadrightarrow R^{\pi,\eta}$ and let 
$a\in I$. The image of $a$ in $R^{\pi,\eta}[1/p]$ is zero, hence $a$ maps to $0$ in $R^{\omega\eta}[1/p]$. Since 
$R^{\omega\eta}$ is $p$-torsion free, the map $R^{\omega\eta}\rightarrow R^{\omega\eta}[1/p]$ is injective, and hence the image 
of $a$ in $R^{\omega\eta}$ is zero. So the surjection $R^u\twoheadrightarrow R^{\omega\eta}$ factors through 
$R^{\pi,\eta}\twoheadrightarrow R^{\omega\eta}$. Let $\MM_{\pi,\eta}$ and $\MM_{\omega\eta}$ be the maximal ideals in  $R^{\pi,\eta}$ 
and $R^{\omega\eta}$ respectively. Then we obtain a surjection:
\begin{equation}\label{surjectionD}
D^{\pi, \eta}(\FF[\varepsilon])^*\cong \frac{\MM_{\pi, \eta}}{\varpi_L R^{\pi,\eta}+\MM^2_{\pi,\eta}}\twoheadrightarrow 
\frac{\MM_{\omega\eta}}{\varpi_L R^{\omega\eta}+\MM^2_{\omega\eta}}\cong D^{\omega\eta}(\FF[\varepsilon])^*,
\end{equation}
where $\FF[\varepsilon]$ is the dual numbers, $\varepsilon^2=0$, and star denotes $\FF$-linear dual. It follows from
\eqref{surjectionD} that $\dim_{\FF}D^{\pi, \eta}(\FF[\varepsilon])\ge \dim_{\FF}D^{\omega\eta}(\FF[\varepsilon])=3$. 
Now $D^u(\FF[\varepsilon])\cong \Ext^1_{\FF[\GG_{\Qp}]}(\rho, \rho)$, \cite[\S22]{mazur} and so 
$D^{\pi, \eta}(\FF[\varepsilon])$ is isomorphic 
to the image of $\Ext^1_{G, \zeta}(\pi, \pi)$ in $\Ext^1_{\FF[\GG_{\Qp}]}(\rho, \rho)$ via \eqref{VExt}, where 
$\Ext^1_{G,\zeta}(\pi,\pi)$ is Yoneda $\Ext$ in the category of smooth $\FF$-representations of $G$ with central 
character $\zeta$.
Now, \cite[VII.5.3]{col1} implies that the map $\Ext^1_{G, \zeta}(\pi, \pi)\rightarrow \Ext^1_{\FF[\GG_{\Qp}]}(\rho, \rho)$ is an 
injection. We obtain:

\begin{thm} Let $\pi$ be as above and assume that $(p,r)\neq (3,1)$ then $\dim_{\FF}\Ext^1_{G,\zeta}(\pi,\pi)\ge 3$.
\end{thm}

\end{document}